\documentclass[11pt, reqno, letterpaper]{amsart}
\usepackage{preamble_article_ECR}

\usepackage[margin=1in]{geometry}
\title{The endoscopic character identity for even special orthogonal groups}
\date{}
\author{Hao Peng}
\begin{document}



\maketitle

\begin{abstract}
We establish the endoscopic character identity for certain bounded $A$-packets of non-quasisplit even special orthogonal groups, with respect to elliptic endoscopic triples. The proof reduces the non-quasisplit case to the quasisplit case and the real Adams--Johnson case by combining the local-global compatibility principle with Arthur's multiplicity formula for non-quasisplit global even special orthogonal groups established by Chen and Zou~\cite{C-Z24}. This result plays a key role in the author's work~\cite{Pen25a} on the compatibility between the Fargues--Scholze local Langlands correspondence and the classical local Langlands correspondence for even special orthogonal groups.
\end{abstract}

\tableofcontents

\section{Introduction}

In his monumental book~\cite{Art13}, Arthur established the local Langlands correspondence for symplectic and quasisplit special orthogonal groups over local fields $K$ of characteristic zero (the endoscopic classification), along with a description of the automorphic discrete spectrum for these groups over number fields (Arthur's multiplicity formula), via the stable trace formula and the theory of endoscopy. However, for quasisplit even special orthogonal groups, the classification is given only for irreducible representations of $\SO(V)$ up to conjugation by $\bx O(V)$, rather than by $\SO(V)$ itself. Similarly, over a number field $F$, Arthur's multiplicity formula does not distinguish between a square-integrable automorphic representation $\pi$ and its twist by the outer automorphism associated with an element of $\bx O(\mbf V)\setm \SO(\mbf V)$.

In~\cite{C-Z21} and~\cite{C-Z21b}, Chen--Zou extended Arthur's endoscopic classification to non-quasisplit even special orthogonal groups. In~\cite{C-Z24}, they further extended Arthur's multiplicity formula to non-quasisplit even special orthogonal groups. (Strictly speaking, their results are stated for non-quasisplit even orthogonal groups, but this implies the corresponding result for special orthogonal groups via the argument of Atobe--Gan~\cite{A-G17}). Their approach relies on local (respectively, global) theta correspondences between even orthogonal groups and symplectic groups. As in other cases where the local Langlands correspondence was established via theta correspondence (e.g.~\cite{G-T11}, \cite{G-S12}), they did not establish the conjectural endoscopic character identity for the $L$-packets (respectively, ``theta packets'') as formulated in~\cite{Kal16a}.

The main goal of this short paper is to verify the conjectural endoscopic character identity for certain local $A$-parameters that are bounded on the Weil group. As in the quasisplit case~\cite[Theorem 2.2.1]{Art13}, the character identity holds only up to outer automorphism. Let $V$ be a quadratic space of dimension $2n$ over a \nA local field $K$ of characteristic zero, and let $G=\SO(V)$ be the associated special orthogonal group. Suppose we are given a (bounded) $A$-parameter $\psi$ of $G$ and an elliptic endoscopic triple $\mfk e=(G^{\mfk e}, s^{\mfk e}, \xi^{\mfk e})$ for $G$ (see~\S\ref{endodslsimisemifs}), together with an $A$-parameter $\psi^{\mfk e}$ of $G^{\mfk e}$ such that  $\xi^{\mfk e}\circ\psi^{\mfk e}=\psi$. Then $G^{\mfk e}$ is a product of two (possibly trivial) quasisplit special orthogonal groups $G_1=\SO(V_1)$ and $G_2=\SO(V_2)$ over $K$. Arthur~\cite{Art13} assigns to $\psi^{\mfk e}$ an $A$-packet $\tilde\Pi_{\psi^{\mfk e}}(G^{\mfk e})$ consisting of $\bx O(V_1)\times \bx O(V_2)$-conjugacy classes of irreducible unitarizable representations of $G^{\mfk e}(K)$. Similarly, in~\cite{C-Z21b}, a ``theta packet'' $\tilde\Pi_\psi(G)$ is assigned to $\psi$ via theta correspondence, consisting of $\bx O(V)$-conjugacy classes of irreducible unitarizable representations of $G(K)$. For notational uniformity, we will also refer to them as $A$-packets throughout this paper.

Let $G^*$ be the unique quasisplit inner form of $G$. Then $G$ can be realized as a pure inner twist $(\varrho, z)$ of $G^*$. Fix an additive character $\uppsi_K$ of $K$ and a pinning $(B^*, T^*, \{X^*_\alpha\}_{\alpha\in\Delta})$ of $G^*$, which together determine a Whittaker datum $\mfk w$ for $G^*$, cf. ~\cite[\S 5.3]{K-S99}. The Whittaker datum $\mfk w$ and the pure inner twist $(\varrho, z)$ determine a canonical map
\begin{equation*}
\iota_{\mfk w, z}: \tilde\Pi_\psi(G)\to\Irr(\mfk S_\psi), \quad \text{where }\mfk S_\psi\defining\pi_0(Z_{\hat G}(\Im(\psi))),
\end{equation*}
as defined in~Theorem~\ref{psmisemifniefms}.

There exists an outer automorphism $\varsigma$ of $G^*$ that preserves $\mfk w$. In fact, $\varsigma$ can be realized as an element of the corresponding orthogonal group
with determinant $-1$, cf. ~\cite[p.~847]{Tai19}. Via the isomorphism $\varrho$, the element $\varsigma$ induces a rational outer automorphism of $G$, cf. ~\cite[Lemma 9.1.1]{Art13}. Let $\mcl H(G)$ denote the space of smooth compactly supported complex-valued functions on $G(K)$. Following Arthur, let $\tilde{\mcl H}(G)$ denote the subspace of $\mcl H(G)$ consisting of $\varsigma$-invariant functions on $G(K)$. Irreducible smooth representations of $\tilde{\mcl H}(G)$ then correspond to $\bx O(V)$-conjugacy classes of irreducible admissible representations of $G(K)$. Similarly, define $\tilde{\mcl H}(G^{\mfk e})\defining\tilde{\mcl H}(G_1)\times \tilde{\mcl H}(G_2)$ as a subspace of $\mcl H(G^{\mfk e})$.

Fix Haar measures on $G(K)$ and $G^{\mfk e}(K)$. For each $\bx O(V)$-conjugacy class $\tilde\pi$ of admissible irreducible representations of $G(K)$, define the character distribution $\tilde\Theta_{\tilde\pi}$ of $\tilde\pi$ as follows. Choose a representative $\pi$ of $\tilde\pi$, and define $\tilde\Theta_{\tilde\pi}$ as the restriction of the distribution $\Theta_\pi$ to the subspace $\tilde{\mcl H}(G)$. This is independent of the representative $\pi$ chosen. Arthur~\cite{Art13} defined a distribution
\begin{equation*}
\tilde\Theta_{\psi^{\mfk e}}=\sum_{\tilde\pi\in\tilde\Pi_{\psi^{\mfk e}}(G^{\mfk e})}\bra{\tilde\pi, s_{\psi^{\mfk e}}}\tilde\Theta_{\tilde\pi},
\end{equation*}
which is well-defined for $f^{G^{\mfk e}}\in\tilde{\mcl H}(G^{\mfk e})$. Here $\bra{\tilde\pi, -}$ denotes the character of $\mfk S_\psi$ defined by Arthur \cite{Art13}*{Theorem 2.2.1}, and $s_{\psi^{\mfk e}}$ is the image of $-1\in\SL(2, \bb C)$ under $\psi^{\mfk e}$. The distribution $\tilde\Theta_{\psi^{\mfk e}}$ is stable; that is, $\tilde\Theta_{\psi^{\mfk e}}(f^{G^{\mfk e}})=0$ whenever all stable orbital integrals of $f^{G^{\mfk e}}$ vanish.

We prove the endoscopic character identity for a certain class of $A$-parameters. Let $\psi$ be an $A$-parameter of $G$ such that  
\begin{equation*}
\psi=\sum_i\psi_i\boxtimes\sp_{b_i}: W_K\times \SU(2)\times \SL(2, \bb C)\to \LL G
\end{equation*}
as a representation of $\WD_K\times \SL(2, \bb C)$, where $b_i\ne b_j$ for $i\ne j$. We say that $\psi$ is \tbf{good} if each $b_i$ is odd and each $\dim(\psi_i)$ is even, see~\S\ref{insiemefies}. In particular, every tempered $A$-parameter (i.e. a tempered $L$-parameter) is good. The following is the desired endoscopic character identity for $G$.

\begin{mainthm}\label{thm main}
Suppose $\psi$ is a good $A$-parameter. Then
\begin{equation*}
\tilde\Theta_{\psi^{\mfk e}}(f^{G^{\mfk e}})=\sum_{\tilde\pi\in\tilde\Pi_\psi(G)}\iota_{\mfk w, z}(\tilde\pi)(s^{\mfk e}s_\psi)\cdot\tilde\Theta_{\tilde\pi}(f)
\end{equation*}
for any pair of $\Delta[\mfk w, \mfk e, z]$-matching functions $f\in \tilde{\mcl H}(G)$ and $f^{G^{\mfk e}}\in \tilde{\mcl H}(G^{\mfk e})$ (see~\textup{\S\ref{isiseieieniws}}). Here, $\Delta[\mfk w, \mfk e, z]$ denotes the transfer factor between $G$ and $G^{\mfk e}$ defined by Kaletha~\cite{Kal16a}.
\end{mainthm}
\begin{rem}
There is no appearance of the Kottwitz sign $e(G)$~\cite{Kot83}, since it is always equal to $1$ for even special orthogonal groups over $K$. 
\end{rem}

The main theorem was established by Arthur~\cite[Theorem 2.2.1]{Art13} in the quasisplit case. This paper deals with the non-quasisplit case. Our strategy is to reduce to the quasisplit case via the local–global compatibility principle. Suppose $\psi^{\mfk e}$ is discrete and good (see~\S\ref{insiemefies}). We globalize the pair $(G, G^{\mfk e})$ over a totally real number field $F$, with $\mbf G$ quasisplit at all finite places except one (where its localization is $G$); such that both $\mbf G(F\otimes\bb R)$ and $\mbf G^{\mfk e}(F\otimes\bb R)$ admit discrete series.

We globalize the local parameters $\psi^{\mfk e}$ and $\psi$ to suitable elliptic global $A$-parameters of Adams--Johnson type at Archimedean places. We then apply the multiplicity formula established by Chen--Zou~\cite{C-Z24}. Combined with Arthur's stable multiplicity formula~\cite{Art13}, this implies that the endoscopic character identity for $G$ follows if it holds for all other localizations of $\mbf G$.

In the general case where $\psi^{\mfk e}$ is not discrete, we apply parabolic descent to reduce to the discrete case. Similar ideas were employed in the study of endoscopic character identities for inner forms of $\GSp(4)$~\cite{C-G15} and metaplectic groups~\cite{Luo20} over \nA local fields.

\begin{rem}
The results of this paper are formulated for good $A$-parameters. It would be very desirable to remove this restriction and treat general $A$-parameters. The reason for imposing this hypothesis is that, for more general parameters, the globalized parameter may factor through the endoscopic group $\mbf G^{\mfk e}$ in more than one way, making it difficult to isolate the contribution of the prescribed local parameter $\psi^{\mfk e}$. Thus this restriction is not merely a matter of exposition; rather, removing it appears to require additional ideas beyond the present local-global argument. For this reason, we do not pursue this generalization here.
\end{rem}

\begin{rem}
It is natural to ask for an analogue of the main result formulated for $\bx O(2n)$, in view of the formulation of Arthur's classification and subsequent works from that perspective, cf.~\cite{A-G17, AGIKMS}. We have chosen to work with $\SO(2n)$, since this is the form needed for the applications considered in this paper. An extension to $\bx O(2n)$ does not seem to be a purely formal reformulation of our argument; rather, it amounts to working in a formulation closely related to the corresponding twisted endoscopic character relations for $\SO(2n)$, and would require additional input not developed here. We therefore do not pursue this question in the present paper.
\end{rem}

The main result of this paper has been applied in the author's earlier work~\cite{Pen25a}, which establishes the compatibility between the Fargues--Scholze and classical local Langlands correspondences for even special orthogonal groups. Furthermore, using that compatibility, the author upgraded the ambiguous local Langlands correspondence to an ``unambiguous'' version---that is, not merely defined up to outer automorphism, cf.~\cite[\S~7.1]{Pen25a}.

\subsection*{Acknowledgements}

The author is indebted to Rui Chen and Jialiang Zou for introducing him to the local Langlands correspondence for even orthogonal groups. He is deeply grateful to his advisor, Prof. Wei Zhang, for his detailed feedback and many valuable suggestions. He also thanks Prof. Hiraku Atobe for encouraging him to consider the main result for bounded $A$-parameters, rather than restricting to the tempered case. Finally, he thanks the anonymous referees for their careful reading and numerous helpful suggestions.

\subsection{Notation and conventions}

\begin{note}\enskip
\begin{itemize}
\item
Let $\bb N$ denote the set of non-negative integers and $\bb Z_+$ the set of positive integers. As usual, we write $\bb Z$, $\bb Q$, $\bb R$, and $\bb C$ for the integers, rational numbers, real numbers, and complex numbers, respectively.
\item
Suppose $X$ is a set.
\begin{itemize}
\item
Let $\uno\in X$ denote the distinguished trivial element (this notation is only used when the notion of triviality is clear from context).
\item
Let $\#X$ be the cardinality of $X$.
\item
For $a, b\in X$, we define the Kronecker symbol
\begin{equation*}
\delta_{a, b}\defining\begin{cases} 1 &\If a=b,\\ 0 &\If a\ne b.\end{cases}
\end{equation*}
\end{itemize}
\item
For a finite group $A$, write $\Irr(A)$ for the set of isomorphism classes of its irreducible complex representations.
\item
For a perfect field $K$, write $\Gal_K$ for its absolute Galois group $\Gal(\ovl K/K)$.
\item
If $K$ is a local field of characteristic zero, write $W_K$ for its Weil group and $\WD_K$ for its Weil--Deligne group:
\begin{equation*}
\WD_K =
\begin{cases}
W_K & \text{if } K = \bb R \text{ or } \bb C, \\
W_K \times \SU(2) & \text{otherwise}.
\end{cases}
\end{equation*}
\item
If $K$ is non-Archimedean, let $I_K$ denote its inertia group, and $\largel{-}_K$ the normalized absolute value on $K^\times$. Via the Artin map $\Art_K: K^\times\xr\sim W_K^\ab$, we extend $\largel{-}_K$ to a character $W_K \to \bb R_+$.
\item
Suppose $F$ is a number field.
\begin{itemize}
\item
Let $\fPla_F$ denote the set of finite places of $F$, and $\infPla_F\defining\Hom(F, \bb C)$ the set of infinite places; set $\Pla_F = \fPla_F \cup \infPla_F$.
\item
Write $\Ade_F$ for the ring of adeles of $F$, $\Ade_{F,f}$ for the finite adeles, and $\mbf C_F$ for the idele class group.
\end{itemize}
\item
Suppose $K$ is a local field and $\msf G$ is a connected reductive group over $K$. Let $\mcl H(\msf G)$ denote the test function space on $\msf G(K)$:
\begin{equation*}
\mcl H(\msf G(K))=\begin{cases}
C_c^\infty(\msf G(K)) & \text{if } K \text{ is non-Archimedean},\\
\{\,f\in C_c^\infty(\msf G(\bb R)):\ f \text{ is bi-\(\mdc K\)-finite}\,\} & \text{if } K=\bb R,
\end{cases}
\end{equation*}
where $\mdc K\subset \msf G(K)$ is a fixed maximal compact subgroup.
\end{itemize}
\end{note}

\section{Local \texorpdfstring{$A$}{A}-packets}\label{lcoaliianisneiifs}

In this section, we recall properties of (ambiguous) $A$-packets following \cites{A-G17, C-Z21}.

\subsection{The groups}\label{sieiiwiwiwms}

Suppose $K$ is a local or global field of characteristic zero. Let $V$ be a $2n$-dimensional $K$-vector space (with $n\in\bb Z_+$), equipped with a non-degenerate symmetric bilinear form $\bra{-, -}$, that is,
\begin{equation*}
\bra{au+bv, w}=a\bra{u, w}+b\bra{v, w}, \quad\text{and}\quad\bra{v, w}=\bra{w, v}
\end{equation*}
for all $a, b\in K$ and $u, v\in V$. 

Fix an orthogonal basis $\{v_1, \ldots, v_{2n}\}$ of $V$ over $K$ such that  $\bra{v_i, v_i}=a_i\in K^\times$. Define the discriminant of $V$ as 
\begin{equation*}
\disc(V)=(-1)^n\prod_{i=1}^{2n}a_i
\end{equation*}
whose class in $K^\times/(K^\times)^2$ is independent of the choice of orthogonal basis. 

If $K$ is local, we define the (normalized) Hasse--Witt invariant 
\begin{equation*}
\eps(V)=\paren{-1, (-1)^{\binom{n}{2}}\cdot\disc(V)^{n-1}}_K\cdot\prod_{1\le i<j\le 2n}(a_i, a_j)_K
\end{equation*}
where
\begin{equation*}
(-, -)_K:\paren{K^\times/(K^\times)^2}\times\paren{K^\times/(K^\times)^2}\to \Br(K)[2]\cong\{\pm1\}
\end{equation*}
denotes the Hilbert symbol. 

We set $G\defining \SO(V)$. If $K$ is a non-Archimedean local field, then
\begin{itemize}
\item
$G$ is split if $\disc(V)=1, \eps(V)=1$;
\item
$G$ is non-quasisplit if $\disc(V)=1$ and $\eps(V)=-1$;
\item
$G$ is quasisplit but non-split if $\disc(V)\ne 1$.
\end{itemize}

As in \cite[Remark 2.4.(2)]{Bor79}, we denote by $\LL G$ the $L$-group of $G$ with component group $\Gal(K(\sqrt{\disc(V)})/K)$, and let $\hat G$ denote its identity component. Then $\hat G=\SO(2n, \bb C)$, and there is a canonical identification
\begin{equation}
\LL G=\begin{cases}
\SO(2n, \bb C) & \If \alpha=1,\\
\bx O(2n, \bb C) &\If \alpha\ne 1.
\end{cases}
\end{equation}
Let $\hat\Std: \LL G\to \GL_{2n}(\bb C)$ be the standard representation. For later use, for each $\alpha\in K^\times/(K^\times)^2$, we denote by $\SO_{2n}^\alpha$ the quasisplit special orthogonal group associated with some quadratic space over $K$ with discriminant $\alpha$.

We fix a pinning $(B^*, T^*, \{X^*_\alpha\}_{\alpha\in\Delta})$ of $G^*$ by identifying it with $\SO(V^*)$ for a suitable quadratic space $V^*$ over $K$, and choosing a complete flag of totally isotropic subspaces in $V^*$. Recall that a Whittaker datum for $G^*$ is a $T^*(K)$-conjugacy class of generic characters of $N^*(K)$, where $N^*$ is the unipotent radical of $B^*$. The set of Whittaker data for $G^*$ forms a principal homogeneous space over the finite Abelian group
\begin{equation*}
E=\coker\paren{G^*(K)\to G^*_\ad(K)}=\ker(H^1(K, Z(G^*))\to H^1(K, G^*)),
\end{equation*}
cf. ~\cite[\S 9]{GGP12}. The fixed pinning $(B^*, T^*, \{X^*_\alpha\}_{\alpha\in\Delta})$ of $G^*$, together with the additive character $\psi_K$ of $K$, determines a Whittaker datum $\mfk w$ for $G^*$, cf. ~\cite[\S 5.3]{K-S99}. If $K$ is non-Archimedean and $G$ is unramified, there exists a unique $G(K)$-conjugacy class of hyperspecial maximal compact open subgroups compatible with $\mfk w$, in the sense of \cite{C-S80}. In this case, ``unramified representations of $G(K)$'' refers to those unramified with respect to such a hyperspecial subgroup.


We may realize $G$ as a pure inner twist of $G^*$; that is, there exists an isomorphism $\varrho: G^*_{\ovl K}\to G_{\ovl K}$ and a $1$-cocycle $z\in Z^1(\Gal_K, G^*)$ such that
\begin{equation*}
\varrho^{-1}\sigma(\varrho)=\Ad(z(\sigma))\quad\text{ for all }\sigma\in \Gal_K. 
\end{equation*}

For each parabolic pair $(M, P)$ of $G$, there exists a unique standard parabolic pair $(M^*, P^*)$ of $G^*$ corresponding to $(M, P)$ under $\varrho$, and this determines an equivalence class of pure inner twist of $M^*$, which we also denote by $(\varrho, z)$ by abuse of notation.

Following Arthur \cite{Art13}, we define $\OAut_N(G^*)\defining\bx O(2n, \bb C)/\SO(2n, \bb C)$, which is a subgroup of cardinality 2 of the outer automorphism group of $\hat G$.\footnote{Here the subscript $N$ refers to the standard representation of the $L$-group of $G^*$ into $\GL_N(\bb C)$, under which $G^*$ is considered in the endoscopic framework for $\GL_N$.} It can be identified with the group of automorphisms of $G^*$ preserving the fixed pinning. Let $\varsigma$ denote the nontrivial element in $\OAut_N(G^*)$. There is a rational action of $\varsigma$ on $G$ via the inner twist $\varrho$, cf. ~\cite[Lemma 9.1.1]{Art13}.

\subsection{Local $A$-parameters and $A$-packets}\label{insiemefies}

In this subsection, we assume that $K$ is a local field. For any connected reductive group $\msf G$ over $K$, we let $\Psi(\msf G)$ denote the set of $\hat{\msf G}$-conjugacy classes of $A$-parameters
\begin{equation*}
\psi: \WD_K\times\SL(2, \bb C)\to \LL\msf G
\end{equation*}
such that $\psi(W_K)$ is a precompact subset of the target. Here $\LL\msf G=\hat{\msf G}\rtimes W_K$ is the Langlands $L$-group in the reduced Weil form, as defined in \cite[Remark 2.4.(2)]{Bor79}.

Let $\msf G^*$ be the unique quasisplit inner form of $\msf G$. Then the $L$-groups of $\msf G$ and $\msf G^*$ are canonically identified, so we identify $\Psi(\msf G)$ with $\Psi(\msf G^*)$.

Let $m\in\bb Z_+$. An $A$-parameter for $\GL(m)_K$ can be regarded as an isomorphism class of $m$-dimensional representations 
\begin{equation*}
\psi: \WD_K\times \SL(2, \bb C)\to \GL(m, \bb C).
\end{equation*}
Every such representation is isomorphic to a finite direct sum of representations of the form $\rho\boxtimes S_a\boxtimes \sp_b$, where
\begin{itemize}
\item
$\rho$ is a smooth representation of $W_K$,
\item
$S_a$ is the unique irreducible representation of $\SU(2)$ of dimension $a$ (Here we assumed that $a$ must equal to 1 if $K$ is Archimedean); and
\item
$\sp_b$ is the unique irreducible algebraic representation of $\SL(2, \bb C)$ of dimension $b$.
\end{itemize}

An $A$-parameter $\psi$ for $\GL(m)$ over $K$ is said to be self-dual and irreducible if $\psi\cong \psi^\vee$ as irreducible representations of $\WD_K\times\SL(2, \bb C)$. Following~\cite[\S 3]{GGP12}, we define the sign of a self-dual irreducible $A$-parameter $\psi$: There exists an isomorphism $f: \psi\cong \psi^\vee$ such that $f^\vee=b(\psi)f$ for some $b(\psi)\in\{\pm1\}$. The value $b(\psi)$ is independent of the choice of $f$, and is called the sign of $\psi$. If $\psi=\rho\boxtimes S_a\boxtimes\sp_b$, where $\rho$ is an irreducible representation of $W_K$, then
\begin{equation*}
b(\psi)=b(\rho)(-1)^{a+b},
\end{equation*}
cf. ~\cite[Lemma 3.2]{GGP12}, \cite[\S 1.2.4]{KMSW}.

We now restrict to the case $G=\SO(V)$. By~\cite[Theorem 8.1]{GGP12} and~\cite[p. 365]{A-G17}, there is a natural identification
\begin{equation*}
\Psi(G^*) = \left\{
\begin{aligned}
&\text{admissible }\psi: \WD_K \times \SL(2, \bb C) \to \bx O(2n, \bb C) \\
&\text{such that } \det(\psi) = (\Art_K^{-1}(-), \disc(V))_K
\end{aligned}
\right\} / \SO(2n, \bb C).
\end{equation*}
Here $\psi$ is called admissible if 
\begin{enumerate}
\item
$\psi(\Frob_K)$ is semisimple, where $\Frob_K$ is a geometric Frobenius element in $W_K$;
\item
$\psi|_{I_K}$ is smooth;
\item
$\psi(W_K)$ is a precompact subset of the target;
\item
$\psi|_{\SU(2)\times \SL(2, \bb C)}$ is algebraic.
\end{enumerate}
The group $\bx O(2n, \bb C)$ acts on $\Psi(G^*)$ by conjugation. Let $\tilde\Psi(G^*)$ denote the set of orbits, and write $\tilde\psi\in\tld\Psi(G^*)$ for the image of $\psi\in \Psi(G^*)$ in $\tilde\Psi(G^*)$.

Any $A$-parameter $\psi\in\Psi(G^*)$ gives rise, via the standard representation $\hat\Std$, to a self-dual representation $\psi^\GL$ of $\WD_K\times\SL(2, \bb C)$ of dimension $2n$. The parameter $\psi$ is determined by $\psi^\GL$ up to $\bx O(2n, \bb C)$-conjugacy, cf. ~\cite[Theorem 8.1]{GGP12}. Writing
\begin{equation*}
\psi^\GL=\sum_im_i\rho_i\boxtimes S_{a_i}\boxtimes\sp_{b_i}
\end{equation*}
as representations of $\WD_K\times\SL(2, \bb C)$, we define the following conditions: 
\begin{itemize}
\item
$\psi$ is called \tbf{tempered} (or an $L$-parameter) if $\psi|_{\SL(2, \bb C)}=\uno$;
\item
$\psi$ is called \tbf{discrete} if each $\rho_i$ is self-dual and $m_i=1$;
\item
$\psi$ is called a \tbf{simple supercuspidal $L$-parameter} if it is tempered, discrete, trivial on the $\SU(2)$-component, and irreducible as a $W_K$-representation;
\item
$\psi$ is called \tbf{elementary} if $m_i=1$ and $1\in\{a_i, b_i\}$ for all $i$;
\item
$\psi$ is called \tbf{good} if all $b_i$ are odd and $\psi$ can be written as
\begin{equation*}
\psi^\GL=\sum_j\phi_j\boxtimes\sp_{b_j},
\end{equation*}
where each $\phi_j$ is a representation of $\WD_K$ of even dimension. In particular, any tempered $L$-parameter is good.
\end{itemize}

We write $\psi^\GL$ as
\begin{equation*}
\psi^\GL=\bplus_{i\in I_\psi^+}m_i\psi_i\oplus \bplus_{i\in I_\psi^-}2m_i\psi_i\oplus\bplus_{i\in J_\psi}m_i(\psi_i\oplus\psi_i^\vee),
\end{equation*}
where $m_i\in\bb Z_+$, and $I_\psi^+, I_\psi^-, J_\psi$ index mutually inequivalent irreducible representations of $\WD_K\times\SL(2, \bb C)$ such that 
\begin{enumerate}
\item
For $i\in I_\psi^+$, $\psi_i$ is self-dual with sign $+1$.
\item
For $i\in I_\psi^-$, $\psi_i$ is self-dual with sign $-1$.
\item
For $i\in J_\psi$, $\psi_i$ is not self-dual.
\end{enumerate}
Then $\psi$ is discrete if and only if $m_i=1$ for all $i$, and $I_\psi^-=J_\psi=\vn$.

For any $\psi\in\Psi(G^*)$, we define
\begin{equation*}
S_\psi^\sharp\defining\prod_{i\in I_\psi^+}\bx O(m_i, \bb C)\times\prod_{i\in I_\psi^-}\Sp(2m_i, \bb C)\times\prod_{i\in J_\psi}\GL(m_i, \bb C),
\end{equation*}
which formally represents the centralizer of $\psi$ in $\LL G$, and its formal component group
\begin{equation*}
\mfk S_\psi^\sharp\defining \pi_0(S_\psi^\sharp)\cong \bplus_{i\in I_\psi^+}(\bb Z/2)e_i,
\end{equation*}
where each $e_i$ is a formal place-holder. There is a canonical identification
\begin{equation*}
\Irr(\mfk S_\psi^\sharp)=\bplus_{i\in I_\psi^+}(\bb Z/2)e_i^\vee, \quad\text{where}\quad e_i^\vee(e_j)=\delta_{i, j}.
\end{equation*}

We also define
\begin{equation*}
S_\psi\defining Z_{\hat G}(\psi)
\end{equation*}
which is naturally a subgroup of $S_\psi^\sharp$ of index at most two. There is a homomorphism
\begin{equation*}
\det\nolimits_\psi: \mfk S_\psi^\sharp\to \bb Z/2, \quad \sum_{i\in I_\psi^+}x_ie_i\mapsto \sum_{i\in I_\psi^+}x_i\dim(\psi_i),
\end{equation*}
and we define the formal component group $\mfk S_\psi\defining\ker(\det_\psi)= \pi_0(S_\psi)$. We define the central element
\begin{equation*}
z_\psi \defining\sum_{i \in I_\psi^+} m_i e_i \in \mfk S_\psi,
\end{equation*}
and define the reduced component group
\begin{equation*}
\ovl{\mfk S}_\psi \defining \mfk S_\psi /\bra{z_\psi}.
\end{equation*}
We denote by $s_\psi$ the image of $-1\in\SL(2, \bb C)$ under $\psi$, which lies in $S_\psi$.

If $\psi_1$ and $\psi_2$ are $\bx O(2n, \bb C)$-conjugate, then $\psi_1^\GL=\psi_2^\GL$. In particular, the definitions of $\psi^\GL, I_\psi^+, I_\psi^-, J_\psi, S_\psi^\sharp, S_\psi, \mfk S_\psi, \ovl{\mfk S}_\psi, z_\psi$, and $s_\psi$ are well-defined for $\tilde\psi\in \tilde\Psi(G^*)$.

There exists an outer automorphism $\varsigma$ of $G^*$ preserving the Whittaker datum $\mfk w$. It can be realized as an element of the corresponding orthogonal group with determinant $-1$, cf. ~\cite[p.~847]{Tai19}. Via the inner twisting $\varrho$, the element $\varsigma$ induces a rational outer automorphism of $G$, cf. \cite[Lemma 9.1.1]{Art13}. Following Arthur, we define $\tilde{\mcl H}(G)$ as the subspace of $\mcl H(G)$ consisting of $\varsigma$-invariant distributions on $G(K)$. The irreducible smooth representations of $\tilde{\mcl H}(G)$ then correspond to $\bx O(V)$-conjugacy classes of irreducible admissible representations of $G(K)$.

We now assume that $K$ is \na. Let $\psi\in\Psi(G^*)$ be an $A$-parameter. Using theta correspondence, Chen--Zou~\cite{C-Z21b} constructed a finite set $\tilde\Pi_\psi(G)$ of $\bx O(V)$-conjugacy classes of irreducible unitarizable representations of $G$, and a map
\begin{equation*}
\iota_{\mfk w, z}: \tilde\Pi_\psi(G)\to \Irr(\mfk S_\psi).
\end{equation*}
We refer to $\tilde\Pi_\psi(G)$ as the (ambiguous) $A$-packet associated with $\psi$, since both this packet and the map $\iota_{\mfk w, z}$ agree with Arthur's packet~\cite{Art13} when $G$ is quasisplit, see~\cite{C-Z21b}*{Theorem 8.7}. The $A$-packet $\tilde\Pi_\psi(G)$ depends only on the $\bx O(2n, \bb C)$-conjugacy class $\tilde\psi$ of $\psi$, and we may thus write it as $\tilde\Pi_{\tilde\psi}(G)$. Note that, when $G$ is quasisplit, $\tld\Pi_\psi(G)$ is originally defined as a multiset by Arthur, but it is in fact a set. When $G$ is quasi-split, this follows from Moeglin's results, as revisited for example in \cite[Theorem~8.9]{Xu17a}; see also \cite[Theorem~4.2]{C-Z21b}. When $G$ is a non-quasisplit even special orthogonal group, the corresponding multiplicity-freeness is proved in \cite[Lemma~5.1]{C-Z21b}.

We restate their result in the following form.

\begin{thm}\label{psmisemifniefms}
For each $\tilde\psi\in\tilde\Psi(G^*)$, there exists a finite set $\tilde\Pi_{\tilde\psi}(G)$ of $\varsigma$-conjugacy classes of irreducible unitarizable representations of $G(K)$, called the (ambiguous) $A$-packet associated with $\tilde\psi$. These $A$-packets satisfy the following properties:
\begin{enumerate}
\item
If $\tilde\psi\in\tilde\Psi(G^*)$ is tempered, then every $\tilde\pi\in\tilde\Pi_{\tilde\psi}(G)$ is tempered; if $\tilde\psi$ is moreover discrete, then every $\tilde\pi$ lies in the discrete series.
\item
The packet $\tilde\Pi_{\tilde\psi}(G)$ depends only on $G$, and not on the choice of inner twist $(\varrho, z)$. Given the fixed Whittaker datum $\mfk w$ of $G^*$, which induces a Whittaker datum on each standard Levi subgroup of $G^*$, there exists a canonical map
\begin{equation*}
\iota_{\mfk w, z}: \tilde\Pi_{\tilde\psi}(G)\to \Irr(\mfk S_{\tilde\psi}).
\end{equation*}
For each $\eta\in\Irr(\mfk S_{\tilde\psi})$, we write $\tilde\pi_{\mfk w, z}(\tilde\psi, \eta)$ for the direct sum (possibly zero) of elements in the fiber of $\iota_{\mfk w, z}$ over $\eta\in\Irr(\mfk S_{\tilde\psi})$, viewed as a $\tilde{\mcl H}(G)$-module.
\item(Local intertwining relations)
Suppose $\tilde\psi\in\tilde\Psi(G^*)$ satisfies $\tilde\psi^\GL=\psi_\tau+\tilde\psi_0^\GL+\psi_\tau^\vee$, where $\psi_\tau\in\Psi(\GL(d))$ corresponds to an irreducible unitarizable representation $\tau$ of $\GL(d, K)$ of Arthur type, and $\tilde\psi_0\in\tilde\Psi\paren{\SO(2n-2d)^{\disc(V)}}$. Let $P^*\le G^*$ be a standard maximal parabolic subgroup with a Levi factor
\begin{equation*}
M^*\cong \GL(d)\times\SO(2n-2d)^{\disc(V)}.
\end{equation*}
Assume that $P=\varrho(P^*)$ is a parabolic subgroup of $G$, and let $\mfk w_0$ denote the induced Whittaker datum of $\mfk w$ on $M^*$. Then for each $\eta_0\in\Irr(\mfk S_{\tilde\psi_0})$,\begin{equation*}
\bx I_P^G(\tau\boxtimes\tilde\pi_{\mfk w_0, z}(\tld\psi_0, \eta_0))=\bplus_\eta\tilde\pi_{\mfk w, z}(\tilde\psi, \eta),
\end{equation*}
as a $\tilde{\mcl H}(G)$-module, where $\eta$ runs over all characters of $\mfk S_{\tilde\psi}$ that restricts to $\eta_0$ under the natural embedding $\mfk S_{\tilde\psi_0}\inj \mfk S_{\tilde\psi}$.
\end{enumerate}
\end{thm}
\begin{proof}
This theorem follows from~\cite[Theorem A.1]{C-Z21} and~\cite[Corollary 5.4, Theorem 7.3]{C-Z21b}.
\end{proof}

Moreover, by results of Moeglin~\cite{Moe06} and Xu~\cite{Xu17a}, the structure of the $A$-packet is better understood when $\psi$ is elementary.

\begin{lm}\label{isisieiemfiems}
Suppose $(\varrho, z)$ is the trivial pure inner twist (i.e., $\varrho = \id: G^* \to G$ and $z = 1$), so that $G=G^*$. If $\psi$ is an elementary $A$-parameter for $G$, then $\iota_{\mfk w, z}$ induces a canonical bijection
\begin{equation*}
\tilde\Pi_{\tilde\psi}(G)\cong \Irr(\ovl{\mfk S}_{\tilde\psi}).
\end{equation*}
\end{lm}
\begin{proof}
This result follows from~\cite[Theorem~6.1]{Xu17a}.
\end{proof}

\section{Endoscopy theory}\label{endodslsimisemifs}

Let $F$ be a local or global field of characteristic zero. In this section, we recall the theory of endoscopy for even special orthogonal groups over $F$.

\subsection{Endoscopic triples}

We begin by recalling the definition of elliptic endoscopic triples, following~\cites{K-S99, Wal10}. For any connected reductive group $\msf G$ over $F$, an \tbf{extended endoscopic triple} for $\msf G$ is a triple $\mfk e=(\msf G^{\mfk e},\msf s^{\mfk e}, \LL\xi^{\mfk e})$, where $\msf G^{\mfk e}$ is a connected quasisplit reductive group over $F$, $\msf s^{\mfk e}\in \hat{\msf G}$ is a semisimple element, and $\LL\xi^{\mfk e}: \LL{\msf G}^{\mfk e}\to \LL{\msf G}$ is an $L$-embedding such that 
\begin{equation*}
\Ad(\msf s^{\mfk e})\circ \LL\xi^{\mfk e}=\LL\xi^{\mfk e}
\end{equation*}
and $\LL\xi^{\mfk e}(\hat{\msf G^{\mfk e}})$ is a connected component of the subgroup of $\Ad(\msf s^{\mfk e})$-fixed elements of $\hat{\msf G}$.

The triple $\mfk e$ is called \tbf{elliptic} if
\begin{equation*}
\LL\xi^{\mfk e}\paren{Z(\hat{\msf G^{\mfk e}})^{\Gal_F}}^\circ\subset Z(\hat{\msf G}),
\end{equation*}
i.e., the identity component of the Galois-invariant center of $\hat{\msf G^{\mfk e}}$ maps into the center of $\hat{\msf G}$ under $\LL\xi^{\mfk e}$. An isomorphism between two endoscopic triples $\mfk e, \mfk e'$ is an element $\msf g\in \hat{\msf G}$ such that
\begin{enumerate}
\item
$\msf g\LL\xi^{\mfk e}(\LL\msf G^{\mfk e})\msf g^{-1}=\LL\xi^{\mfk e'}(\LL\msf G^{\mfk e'})$, and
\item
$\msf g\msf s^{\mfk e}\msf g^{-1}=\msf s^{\mfk e'}$ modulo $Z(\hat{\msf G})$.
\end{enumerate}

We denote by $\mcl E(\msf G^*)$ the set of isomorphism classes of extended endoscopic triples for $\msf G^*$, and by $\mcl E_\ellip(\msf G^*)$ the subset of elliptic classes.

Suppose $\mfk e\in \mcl E(\msf G^*)$. Then for each $g^{\mfk e}\in\hat{\msf G}^{\mfk e}$, the image $\LL\xi^{\mfk e}(g^{\mfk e})\in\LL G$ defines an automorphism of $\mfk e$. This allows us to define the outer automorphism group of $\mfk e$ as
\begin{equation}\label{equirenrinfids}
\OAut(\mfk e)\defining \Aut(\mfk e)/\LL\xi^{\mfk e}(\hat{\msf G^{\mfk e}}).
\end{equation}

We now specialize to the case $G=\SO(V)$. According to~\cite[\S 1.8]{Wal10}, the isomorphism classes of elliptic extended endoscopic triples for $G$ are determined by the endoscopic group
\begin{equation*}
G^{\mfk e}=\SO(2n_1)^{\mfk d_1}\times\SO(2n_2)^{\mfk d_2},
\end{equation*}
where $n_1+n_2=n$ and $\mfk d_1\mfk d_2=\disc(V)$. Here it is required that $\mfk d_i=1$ when $n_i=0$, and $(n_i,\mfk d_i)\ne (1, 1)$ for each $i\in\{1, 2\}$.

The outer automorphism group $\OAut(\mfk e)$ is trivial unless $n_1n_2\ne 0$, in which case there exists a nontrivial outer automorphism given by simultaneous outer conjugation on both factors of $\hat{G^{\mfk e}}$. Furthermore, if $n_1=n_2$ and $\disc(V)=1$, there exists an additional nontrivial automorphism swapping the two factors. In this case, the full outer automorphism group satisfies
\begin{equation*}
\OAut(\mfk e)\cong \bb Z/2\times \bb Z/2.
\end{equation*}

\subsection{Orbital integrals}

We briefly recall some of the discussion from~\cite{Art13}*{\S 2.1}. Suppose $F=K$ is a local field and let $\msf G$ be a connected reductive group over $K$. We denote by $\msf G(K)_\sreg\subset\msf G(K)$ the open subset of strongly regular semisimple elements, i.e., regular semisimple elements whose centralizer is connected (in fact, a maximal torus).

Fix a Haar measure on $\msf G(K)$. For $\delta\in \msf G(K)_\sreg$, the \tbf{Weyl discriminant} of $\delta$ is defined as 
\begin{equation*}
D^{\msf G}(\delta)\defining \det(1-\Ad(\delta)|\mfk g/\mfk g_\delta)\in K^\times,
\end{equation*}
where $\mfk g$ and $\mfk g_\delta$ are the Lie algebras of $\msf G$ and $Z_{\msf G}(\delta)$, respectively. Fix a Haar measure on the torus $Z_{\msf G}(\delta)$, which induces a quotient measure on $Z_{\msf G}(\delta)(K)\bsh \msf G(K)$.

For $f\in \mcl H(\msf G)$ and $\delta\in \msf G(K)_\sreg$, the \tbf{normalized orbital integral} of $f$ at $\delta$ is defined as
\begin{equation*}
\Orb_\delta(f)\defining\largel{D^{\msf G}(\delta)}^{\frac{1}{2}}\int_{Z_{\msf G}(\delta)(K)\bsh \msf G(K)}f(x^{-1}\delta x)\bx dx.
\end{equation*}
If $\msf G$ is quasisplit, the normalized stable orbital integral of $f$ at $\delta$ is defined by summing over its rational conjugacy class:
\begin{equation*}
\SOrb_\delta(f)\defining\sum_{\delta'}\Orb_{\delta'}(f),
\end{equation*}
where $\delta'$ runs over a set of representatives for the $\msf G(K)$-conjugacy classes of those elements that are $\msf G(\ovl K)$-conjugate to $\delta$.

\subsection{Local transfer}\label{isiseieieniws}

In this subsection, we suppose $F=K$ is a local field and continue with the notations established in \S\ref{sieiiwiwiwms}. In particular, $G=\SO(V)$ is an even special orthogonal group, and $(G, \varrho, z)$ is a pure inner twist of the quasisplit inner form $G^*$. We have fixed a Whittaker datum $\mfk w$ for $G^*$.

Recall the subspace $\tilde{\mcl H}(G)$ of $\mcl H(G)$ consisting of $\varsigma$-invariant test functions. Similarly, for each $\mfk e\in \mcl E_\ellip(G)$, we can define the analogous subspace $\tilde{\mcl H}(G^{\mfk e})=\tilde{\mcl H}(G_1)\times \tilde{\mcl H}(G_2)$ of invariant test functions on $G^{\mfk e}$, because $G^{\mfk e}$ is a product of two (possibly trivial) even special orthogonal groups over $K$.

For each extended endoscopic triple $\mfk e\in \mcl E(G^*)$, a \tbf{transfer factor}
\begin{equation*}
\Delta[\mfk w, \mfk e, z]:G^{\mfk e}(K)_\sreg\times G(K)_\sreg\to \bb C
\end{equation*}
is defined in \cite{Kal16a}*{p.233, Equation (6)}, such that $\Delta[\mfk w, \mfk e, z]$ is a function on stable conjugacy classes of $G^{\mfk e}(K)_\sreg$ and conjugacy classes of $G(K)_\sreg$. With the transfer factor in hand, we can now recall the notion of matching test functions from~\cite[\S 5.5]{K-S99}.

\begin{defi}\label{mathcineinitenisehnis}
Let $f^{G^{\mfk e}}\in\mcl H(G^{\mfk e})$ and $f\in \mcl H(G)$. We say that $f^{G^{\mfk e}}$ and $f$ are ($\Delta[\mfk w, \mfk e, z]$-) \tbf{matching test functions} if
\begin{equation*}
\SOrb_\gamma(f^{G^{\mfk e}})=\sum_{\delta\in G(K)_\sreg/\Conj}\Delta[\mfk w, \mfk e, z](\gamma, \delta)\Orb_\delta(f)
\end{equation*}
for every $\gamma\in G^{\mfk e}(K)_\sreg$. In this case, we will also say that $f^{G^{\mfk e}}$ is a transfer of $f$ to $G^{\mfk e}$.
\end{defi}
\begin{rem}
Since the orbital integrals $\Orb_\delta(f)$ depend on the choices of measures on $G(K)$ and $Z_G(\delta)(K)$, the notion of matching test functions also depends on the choice of Haar measures on $G(K), G^{\mfk e}(K)$, and all tori in $G$ and $G^{\mfk e}$. There exists a way to synchronize the various tori; see~\cite[Remark 5.1.2]{A-K24}.
\end{rem}

We now state a theorem asserting the existence of transfers of test functions. When $K=\bb R$, this is a fundamental result of Shelstad \cites{She82, She08}. When $K$ is \na, it is the culmination of extensive work of many people, including Langlands and Shelstad~\cites{L-S87, L-S90}, Waldspurger~\cites{Wal97, Wal06}, and Ng\^o~\cite{Ngo10}.

\begin{thm}\label{sisienfnieies}
Let $f\in \mcl H(G)$ and let $\mfk e\in \mcl E_\ellip(G)$. Then there exists a transfer $f^{G^{\mfk e}}$ of $f$ to $G^{\mfk e}$. If, in addition, $K$ is \nA and $\mdc K$ is a $\varsigma$-stable hyperspecial maximal compact open subgroup of $G(K)$, then there exists a hyperspecial maximal compact open subgroup $\mdc K^{\mfk e}$ of $G^{\mfk e}(K)$ such that the characteristic function $\uno_{\mdc K^{\mfk e}}$ is a transfer of the characteristic function $\uno_{\mdc K}$ to $G^{\mfk e}$, provided that the Haar measures are chosen so that $\Vol(\mdc K)=\Vol(\mdc K^{\mfk e})=1$.

Moreover, if $f$ lies in $\tilde{\mcl H}(G)$, then there exists a transfer $f^{G^{\mfk e}}$ of $f$ to $G^{\mfk e}$ that is contained in $\tilde{\mcl H}(G^{\mfk e})$.
\end{thm}
\begin{proof}
The final assertion follows from the invariant properties of the transfer factor; see~\cite{Tai19}*{p.~860}.
\end{proof}

Similarly, when $F$ is global and $\mfk e\in \mcl E_\ellip(G)$, a function $f^{G^{\mfk e}}=\otimes_vf^{G^{\mfk e}}_v\in \tilde{\mcl H}(\mbf G^{\mfk e}(\Ade_F))\defining \otimes'_v\tld{\mcl H}(\mbf G^{\mfk e}\otimes_FF_v)$ is called an adelic transfer of $f=\otimes_vf_v\in \tilde{\mcl H}(\mbf G(\Ade_F))\defining \otimes'_v\tld{\mcl H}(\mbf G\otimes_FF_v)$ if for each place $v\in\Pla_F$, $f^{G^{\mfk e}}_v$ is a transfer of $f_v$. The existence of adelic transfer then follows from Theorem~\ref{sisienfnieies}.


\section{Adams--Johnson packets}

In this section, we recall Adams--Johnson parameters and the associated packets studied in~\cite{A-J87}, which play a key role in the globalization process. We continue with the notation introduced in \S\ref{lcoaliianisneiifs}, and assume throughout this section that $K=\bb R$ and $G=\SO(V)$.

An $A$-parameter $\psi: W_{\bb R}\times\SL(2, \bb C)\to \LL G$ is called an \tbf{Adams--Johnson parameter} if its infinitesimal character is $C$-algebraic regular in the sense of~\cite{B-G14}. This is always true if $\psi$ is tempered and discrete, although there are other non-tempered examples as well. An Adams--Johnson parameter is always discrete, but not every discrete $A$-parameter is of Adams--Johnson type. For further details, see~\cite[\S 4.2.2]{Tai17} and~\cite[\S 8.1]{AMR18}.

For each Adams--Johnson parameter $\psi$ of $G$, Adams--Johnson \cite{A-J87} constructed a packet $\Pi_\psi^{\bx{AJ}}(G)$ of irreducible admissible representations of $G(\bb R)$. Using the Whittaker datum $\mfk w$ for $G^*$ and the pure inner twist $(\varrho, z)$, Ta\"ibi~\cite[\S 3.2.2-3.2.3]{Tai19} attached to each element of $\Pi_\psi^{\bx{AJ}}(G)$ a character of $S_\psi=\mfk S_\psi$. We will omit the superscript ``AJ'', since for quasisplit $G$, the works of Arancibia, Moeglin, and Renard~\cite{AMR18} show that the Adams--Johnson packet coincides with Arthur's packet, and the two constructions of the map to $\Irr(\mfk S_\psi)$ coincide.

As in the \nA case, we may define $\varsigma$-equivalence classes of Adams--Johnson parameters, denoted $\tilde\psi$, and for each such $\tilde\psi$, there is a corresponding packet $\tilde\Pi_{\tilde\psi}(G)$ consisting of $\bx O(V)$-conjugacy classes of irreducible admissible representations of $G(\bb R)$. The following endoscopic character identity, proved by Ta\"ibi \cite[Proposition 3.2.6]{Tai19}, plays a central role.

\begin{thm}\label{ismsienfieimfws}
Assume that $G^*$ admits discrete series. Let $\psi$ be an Adams--Johnson parameter of $G$, and let $\mfk e\in \mcl E_\ellip(G)$ be an elliptic extended endoscopic triple. If $\psi^{\mfk e}$ is an $A$-parameter for $G^{\mfk e}$ such that  $\LL\xi^{\mfk e}\circ \psi^{\mfk e}=\psi$, then $\psi^{\mfk e}$ is also an Adams--Johnson parameter. Moreover, for any $f\in \tld{\mcl H}(G)$, a transfer $f^{G^{\mfk e}}$ of $f$ to $G^{\mfk e}$ may be chosen in $\tilde{\mcl H}(G^{\mfk e})$, and the following endoscopic character identity holds:
\begin{equation*}
\sum_{\tilde\pi\in\tilde\Pi_{\tilde\psi}(G)}\iota_{\mfk w, z}(\pi)(s^{\mfk e}s_\psi)\tilde\Theta_{\tilde\pi}(f)=\tilde\Theta_{\tilde\psi^{\mfk e}}(f^{G^{\mfk e}}).
\end{equation*}
\end{thm}

\section{Arthur's multiplicity formula}

Let $F$ be a totally real number field. Let $\mbf V$ be a quadratic space of dimension $2n$ over $F$, and let $\mbf G=\SO(\mbf V)$ be the associated even special orthogonal group over $F$. Denote by $\mbf G^*$ the unique quasisplit inner form of $\mbf G$. We may fix a pure inner twisting $(\varrho, z)$ from $\mbf G^*$ to $\mbf G$, which localizes at each place $v$ of $F$ to a pure inner twist $(\varrho_v, z_v)$ from $\mbf G^*_v$ to $\mbf G_v$. We also fix a Whittaker datum $\mfk w$ for $\mbf G^*$ and the induced rational outer automorphism $\varsigma$ of $\mbf G$, cf.~\S2.1.

For each Archimedean place $\tau$ of $F$, fix a maximal compact subgroup $\mdc K_\tau$ of $\mbf G(F_\tau)$. For all but finitely many finite places $v$ of $F$, the inner twist $(\varrho_v, z_v)$ is trivial, and we fix a hyperspecial subgroup $\mdc K_v$ of $\mbf G(F_v)$ compatible with $\mfk w$. Note that $\mdc K_v$ is invariant under the outer automorphism $\varsigma$. Let $\mcl A(\mbf G)$ denote the space of automorphic forms for $\mbf G$ in the sense of~\cite{B-J79}, with respect to the fixed compact subgroup $\prod_{\tau\in\infPla_F}\mdc K_\tau$. We write $\mcl A_2(\mbf G)$ for the space of square-integrable automorphic forms for $\mbf G$.

Then $\mcl A_2(\mbf G)$ admits a decomposition as a direct sum of irreducible automorphic representations:
\begin{equation*}
\mcl A_2(\mbf G)\cong\bplus_\pi\pi^{m(\pi)},
\end{equation*}
where $\pi=\otimes_v\pi_v$ runs over the irreducible automorphic representations of $\mbf G(\Ade_F)$ occurring in the discrete spectrum, and the nonnegative integer $m(\pi)$ denotes the \tbf{automorphic multiplicity} of $\pi$.

Let $\Irr_\disc(\mbf G(\Ade_F))$ denote the set of irreducible automorphic representations of $\mbf G(\Ade_F)$ occurring in the discrete spectrum, i.e. those with $m(\pi)>0$. We also denote by $\tld{\Irr}_\disc(\mbf G(\Ade_F))$ the set of irreducible $\tilde{\mcl H}(\mbf G(\Ade_F))$-modules appearing in $\mcl A_2(\mbf G)$, where
\begin{equation*}
\tilde{\mcl H}(\mbf G(\Ade_F))\defining\bigotimes_v\nolimits'\tilde{\mcl H}(\mbf G(F_v))
\end{equation*}
is the restricted tensor product of the local Hecke algebras \wrt the characteristic function of $\mdc K_v$ (defined for all but finite many finite places $v$ of $F$).

In order to formulate his multiplicity formula, Arthur circumvented the absence of the hypothetical Langlands group by a surrogate notion of global Arthur--Langlands parameters.

Recall that a global $A$-parameter for $\mbf G$ is a formal finite sum
\begin{equation*}
\bm\psi=\sum_i\bm\psi_i\boxtimes\sp_{d_i}
\end{equation*}
where
\begin{itemize}
\item
each $\bm\psi_i$ is an irreducible self-dual cuspidal automorphic representation of $\GL(n_i; \Ade_F)$ with sign $(-1)^{d_i-1}$, in the sense of~\cite{GGP12}*{p.~94};
\item
$\sp_{d_i}$ denotes the $d_i$-dimensional irreducible representation of $\SL(2, \bb C)$;
\item
the total dimension satisfies $\sum_in_id_i=2n$; and
\item
if $\omega_i$ is the central character of $\bm\psi_i$, then $\prod_i\omega_i^{d_i}=\chi_{\mbf V}$, where $\chi_{\mbf V}$ is the quadratic character of $\mbf C_F$ associated with $\disc(\mbf V)$.
\end{itemize}
The parameter $\bm\psi$ is called \tbf{elliptic} (or \tbf{discrete}) if the pairs $(\bm\psi_i, d_i)$ are pairwise distinct. It is called \tbf{generic} (or \tbf{tempered}) if $d_i=1$ for all $i$.

Given an elliptic global $A$-parameter
\begin{equation*}
\bm\psi=\bm\psi_1\boxtimes\sp_{d_1}+\cdots+\bm\psi_k\boxtimes\sp_{d_k}
\end{equation*}
for $\mbf G$, define the formal extended component group
\begin{equation*}
\mfk S_{\bm\psi}^\sharp\defining \bplus_i(\bb Z/2)e_i,
\end{equation*}
where $e_i$ is a formal coordinate associated with the summand $\bm\psi_i\boxtimes\sp_{d_i}$. Denote the determinant character
\begin{equation*}
\det\nolimits_{\bm\psi}: \mfk S_{\bm\psi}^\sharp\to \bb Z/2, \sum_{1\le i\le k}x_ie_i\mapsto \sum_{1\le i\le k}n_ix_i.
\end{equation*}
Then the formal global component group is defined as the kernel
\begin{equation*}
\mfk S_{\bm\psi}\defining\ker(\det\nolimits_{\bm\psi}).
\end{equation*}
Moreover, define the quotient group
\begin{equation*}
\ovl{\mfk S}_{\bm\psi}\defining \mfk S_{\bm\psi}/\bra{e_1+\cdots+e_k}.
\end{equation*}
Finally, Arthur defines a canonical character $\ve_{\bm\psi}$ of $\mfk S_{\bm\psi}$ as in~\cite[Equation~(1.5.6)]{Art13}.

For each elliptic global $A$-parameter $\bm\psi$ of $\mbf G$ and each place $v\in\Pla_F$, there exists a localization $\tilde{\bm\psi}_v$---an $\bx O(2n, \bb C)$-conjugacy class of local $A$-parameters of $\mbf G_{F_v}$. Moreover, there is a natural map of component groups
\begin{equation*}
\mfk S_{\bm\psi}\to\mfk S_{\tilde{\bm\psi}_v},
\end{equation*}
as described in~\cite{C-Z24}*{p.~7}

Furthermore, associated with each local parameter $\tilde{\bm\psi}_v$, we define an $\bx O(2n, \bb C)$-conjugacy class of $L$-parameters $\tilde\phi_{\tilde{\bm\psi}_v}$ as follows: choose a representative $\bm\psi_v$ of $\tilde{\bm\psi}_v$, and define
\begin{equation*}
\phi_{\bm\psi_v}(w)=\bm\psi_v\paren{w, \begin{bmatrix}\largel{w}^{\frac{1}{2}} &\\  &\largel{w}^{-\frac{1}{2}}\end{bmatrix}},
\end{equation*}
which is an $L$-parameter for $\mbf G_{F_v}$. The conjugacy class $\tilde\phi_{\tilde{\bm\psi}_v}$ is independent of the choice of representative $\bm\psi_v$.

We now recall the ``weak form'' of ``Arthur's multiplicity formula'', as established by Chen--Zou~\cite{C-Z24}. In general, this formula has been established only for elliptic generic global $A$-parameters. However, if $\mbf G$ is either quasisplit or has $F$-rank at most one, then the result holds for all elliptic $A$-parameters. 

\begin{thm}\label{mianieijidinfies}
Suppose $\mbf G$ is either quasisplit or has $F$-rank at most one, and let $\mfk w$ be a Whittaker datum for $\mbf G^*$. If $\bm\psi$ is an elliptic global $A$-parameter for $\mbf G$, we denote by $\mcl A_{2, \bm\psi}(\mbf G)$ the direct sum of irreducible admissible representations $\pi$ in $\mcl A_2(\mbf G)$ \sut the $L$-parameter of $\pi_v$ is $\tilde\phi_{\tilde{\bm\psi}_v}$ for almost every place of $F$. Then there is a natural decomposition of $\tilde{\mcl H}(\mbf G(\Ade_F))$-modules
\begin{equation*}
\mcl A_2(\mbf G)=\bplus_{\mbf\psi}\mcl A_{2, \bm\psi}(\mbf G),
\end{equation*}
where $\mbf\psi$ runs through elliptic global $A$-parameters for $\mbf G$.

Furthermore, for reach elliptic global $A$-parameter $\bm\psi=\bm\psi_1\boxtimes\sp_{d_1}+\cdots+\bm\psi_k\boxtimes\sp_{d_k}$ for $\mbf G$, there exists a natural diagonal map
\begin{equation*}
\Delta: \mfk S_{\bm\psi}\to \mfk S_{\bm\psi, \Ade_F}\defining \prod_{v\in\Pla_F}\mfk S_{\tilde{\bm\psi}_v},
\end{equation*}
and a decomposition of $\tilde{\mcl H}(\mbf G(\Ade_F))$-modules:
\begin{equation*}
\mcl A_{2, \bm\psi}(\mbf G(F)\bsh\mbf G(\Ade_F))\cong m_{\bm\psi}\bplus_{\substack{\eta\in\Irr(\mfk S_{\bm\psi, \Ade_F})\\ \Delta^*(\eta)=\ve_{\bm\psi}}}\tilde\pi_{\mfk w, z}(\eta),
\end{equation*}
where each $\tilde\pi_{\mfk w, z}(\eta)=\otimes_v\tilde\pi_{\mfk w_v, z_v}(\tld\psi_v, \eta_v)$ is the global restricted tensor product of local representations (which is not necessarily irreducible). Moreover, the multiplicity $m_{\bm\psi}$ satisfies $m_{\bm\psi}=1$ unless $n_id_i$ is even for each $i$, in which case $m_{\bm\psi}=2$.
\end{thm}
\begin{proof}
This result follows from~\cite[Theorem~1.5.2]{Art13} when $\mbf G$ is quasisplit, and from~\cite[Theorem~2.1, Theorem 7.7]{C-Z24} when $\mbf G$ has $F$-rank at most one. Although Chen--Zou~\cite{C-Z24} states the theorem for the full orthogonal group $\bx O(\mbf V)$, the argument in~\cite[Chapter 7]{A-G17} implies the corresponding result for the special orthogonal group $\SO(\mbf V)$.
\end{proof}

\section{Stable trace formula}

In this section, we recall the stable trace formula for the global even orthogonal group $\mbf G$.

For each extended endoscopic triple $\mfk e\in \mcl E(\mbf G)$, denote
\begin{equation*}
\iota(\mfk e)=\tau(\mbf G)\tau(\mbf G^{\mfk e})^{-1}\#\OAut(\mfk e)^{-1}\in\bb Q
\end{equation*}
to be the global coefficient introduced by Kottwitz and Shelstad, cf.~\cite{Art13}*{Equation (3.2.4)}. Here $\tau(-)$ denotes the Tamagawa number. For instance, $\iota(\mfk e)=1$ if $\mbf G^{\mfk e}=\mbf G^*$, and $\iota(\mfk e)=1/4$ if $\mbf G^{\mfk e}=\SO(2n_1)^{\mfk d_1}\times \SO(2n_2)^{\mfk d_2}$ with $(n_1, \mfk d_1)\ne (n_2, \mfk d_2)$. 

\begin{thm}[Discrete part of stable trace formula, \cite{Art13}*{p. 135}]\label{stiamienidsiniw}
There is a spectral decomposition
\begin{equation}\label{bisnsiefneifs}
\bx I_{\disc}^{\mbf G}(f)=\sum_{\mfk e\in\mcl E_\ellip(\mbf G)}\iota(\mfk e)\ST_{\disc}^{\mbf G^{\mfk e}}(f^{\mbf G^{\mfk e}})
\end{equation}
for any $f^{\mbf G^{\mfk e}}$ that is matching functions with $f$. Here
\begin{equation*}
\bx I_{\disc}^{\mbf G}(f)=\sum_{\mbf M}\frac{1}{\#W(\mbf G,\mbf M)}\sum_{s\in W(\mbf G,\mbf M)_\reg}\frac{1}{\det(s-1)_{\mfk a_{\mbf M}^{\mbf G}}}\tr\paren{M_{\mbf P}(s, 0)\mcl I_{\disc}^P(0, f)},
\end{equation*}
where $\mbf M$ runs through standard Levi subgroups of $\mbf G$, cf. \cite{Art13}*{Page 124}, and $\ST_\disc^{\mbf G^{\mfk e}}(f^{\mbf G^{\mfk e}})$ is the stable linear form on $\mbf G^{\mfk e}$ appearing in the discrete part of Arthur's stabilization of the trace formula.
\end{thm}
\begin{rem}
We refer the reader to \cite{Art13}*{Chapter 3} for the precise definitions of other terms in the right-hand side of the formula above. For our purpose, we simply note that the term corresponding to $\mbf M=\mbf G$ is
the sum of $m(\pi)\tr(\pi(f))$, where $\pi$ runs over the set of discrete automorphic representations $\pi$ of $\mbf G$ with formal parameter $\psi$.

Note that there are only finitely many isomorphism classes of elliptic endoscopic data unramified outside a given finite set of places of $F$ by \cite{Lan83}. Combining with \cite{Kot86}*{Proposition 7.5}, this implies that there are only finitely many non-zero terms on the right-hand side of Equation~\eqref{bisnsiefneifs}.
\end{rem}

\section{Quasi-split case}\label{lslsieieurheifmeis}

Our main result, Theorem~\ref{thm main}, was established by Arthur~\cite[Theorem 2.2.1]{Art13} when $K$ is non-Archimedean and $G$ is quasisplit. In fact, Arthur's result holds without the assumption that $\psi$ be tempered, and the analogous statement for Archimedean local fields is also available. Moreover, one may deduce the following stable multiplicity formula for any (not necessarily quasisplit) global special orthogonal group $\mbf G=\SO(\mbf V)$ over $F$.

For any elliptic endoscopic triple $\mfk e$ for $\mbf G$, the endoscopic group $\mbf G^{\mfk e}=\mbf G_1\times \mbf G_2$ is a product of two (possibly trivial) even special orthogonal groups over $F$. For each elliptic parameter $\bm\psi$ for $\mbf G$, we write $\Psi(\mbf G^{\mfk e}, \bm\psi)$ for the set of pairs of elliptic parameters $(\bm\psi_1, \bm\psi_2)$ for $\mbf G_1$ and $\mbf G_2$ such that  $\bm\psi=\bm\psi_1+\bm\psi_2$ as a formal sum.

\begin{thm}[Stable multiplicity formula, \cite{Art13}*{Corollary 4.1.3}]\label{stablelmultlieisiifomr}
Let $\mfk e$ be an elliptic endoscopic triple for $\mbf G$, and let $\bm\psi$ be an elliptic $A$-parameter for $\mbf G$. For any function $f^{\mbf G^{\mfk e}}=\otimes_vf^{\mbf G^{\mfk e}}_ v\in\tilde{\mcl H}(\mbf G^{\mfk e})$, we have
\begin{equation}
\ST^{\mbf G^{\mfk e}}_{\disc, \bm\psi}(f^{\mbf G^{\mfk e}})=\sum_{\bm\psi^{\mfk e}\in\Psi(\mbf G^{\mfk e}, \bm\psi)}\frac{m_{\bm\psi^{\mfk e}}\ve_{\bm\psi^{\mfk e}}(s_{\bm\psi^{\mfk e}})}{\#\ovl{\mfk S}_{\bm\psi^{\mfk e}}}\prod_v\tilde\Theta_{\tilde{\bm\psi}^{\mfk e}_v}(f_v^{\mbf G^{\mfk e}}).
\end{equation}
Here $m_{\bm\psi^{\mfk e}}, \#\ovl{\mfk S}_{\bm\psi^{\mfk e}}$ and $\ve_{\bm\psi^{\mfk e}}$ are defined as follows: If $\mbf G^{\mfk e}=\prod_i\mbf G_i$ is a product of special orthogonal groups and $\bm\psi^{\mfk e}=\prod_i\bm\psi_i$ where $\bm\psi_i$ is a global $A$-parameter for $\mbf G_i$ for each $i$, then
\begin{equation*}
m_{\bm\psi^{\mfk e}}\defining\prod_im_{\bm\psi_i}, \quad \#\ovl{\mfk S}_{\bm\psi^{\mfk e}}\defining\prod_i\#\ovl{\mfk S}_{\bm\psi_i}, \quad \ve_{\bm\psi^{\mfk e}}(s_{\bm\psi^{\mfk e}})\defining\prod_i\ve_{\bm\psi_i}(s_{\bm\psi_i}).
\end{equation*}
\end{thm}

As a corollary of Theoerm~\ref{stablelmultlieisiifomr}, we have the following simple stable trace formula for any even special orthogonal group over a totally real number field $F$.

\begin{cor}\label{simplestaleitneirnies}
Suppose $\mbf G$ is either quasisplit or has $F$-rank $\le 1$. Let $\bm\psi$ be an elliptic $A$-parameter of $\mbf G$ such that, at each real place of $F$, its localization is an Adams--Johnson parameter. Then, for any $f=\otimes_vf_v\in\tilde{\mcl H}(\mbf G(\Ade_F))$,
\begin{equation}\label{ismsiniefmeis}
m_{\bm\psi}\sum_{\substack{\otimes_v\eta_v\in\Irr(\mfk S_{\bm\psi, \Ade_F})\\ \Delta^*(\otimes_v\eta_v)=\ve_{\bm\psi}}}\prod_v\tilde\Theta_{\tilde\pi^{\mbf G}_{\eta_v}}(f_v)=\sum_{\mfk e\in\mcl E_\ellip(\mbf G)}\iota(\mfk e)\sum_{\bm\psi^{\mfk e}\in\Psi(\mbf G^{\mfk e}, \bm\psi)}\frac{m_{\bm\psi^{\mfk e}}\ve_{\bm\psi^{\mfk e}}(s_{\bm\psi^{\mfk e}})}{\#\ovl{\mfk S}_{\bm\psi^{\mfk e}}}\prod_v\tilde\Theta_{\tilde{\bm\psi}^{\mfk e}_v}(f_v^{\mbf G^{\mfk e}}),
\end{equation}
Here $f^{\mbf G^{\mfk e}}=\otimes_vf^{\mbf G^{\mfk e}}_ v\in\tilde{\mcl H}(\mbf G^{\mfk e})$ are adelic transfers of $f$.
\end{cor}
\begin{proof}
It follows from Theorem~\ref{stiamienidsiniw} and Arthur's standard procedure on extracting the $\psi$-part of the stable trace formula (cf.  \cite{Art13}*{Chapter 3}) that
\begin{equation}
\bx I^{\mbf G}_{\disc, \bm\psi}(f)=\sum_{\mfk e\in \mcl E_\ellip(\mbf G)}\iota(\mfk e)\ST_{\disc, \bm\psi}^{\mbf G^{\mfk e}}(f^{\mbf G^{\mfk e}}).
\end{equation}
Next, we claim $\bx I^{\mbf G}_{\disc, \bm\psi}(f)$ is the sum of $m(\pi)\tr(\pi(f))$ for $\pi$ running over the set of discrete automorphic representations $\pi$ of $\mbf G$ with global $A$-parameter $\bm\psi$. This is true because the contributions from properly-contained Levi subgroups $\mbf M$ of $\mbf G$ only involve representations of $\mbf G(\Ade_F)$ with non-regular infinitesimal character at all infinite places of $F$.

It then follows from Arthur's multiplicity formula Theorem~\ref{mianieijidinfies} that 
\begin{equation}
\bx I^{\mbf G}_{\disc, \bm\psi}(f)=m_{\bm\psi}\sum_{\substack{\otimes_v\eta_v\in\Irr(\mfk S_{\bm\psi, \Ade_F})\\ \Delta^*(\otimes_v\eta_v)=\ve_{\bm\psi}}}\prod_v\tilde\Theta_{\tilde\pi^{\mbf G}_{\eta_v}}(f_v).
\end{equation}
On the other hand, it follows from Theorem~\ref{stablelmultlieisiifomr} that
\begin{equation}
\ST^{\mbf G^{\mfk e}}_{\disc, \bm\psi}(f^{\mbf G^{\mfk e}})=\sum_{\psi^{\mfk e}\in\Psi(\mbf G^{\mfk e}, \bm\psi)}\frac{m_{\bm\psi^{\mfk e}}\ve_{\bm\psi^{\mfk e}}(s_{\bm\psi^{\mfk e}})}{\#\ovl{\mfk S}_{\bm\psi^{\mfk e}}}\prod_v\tilde\Theta_{\tilde{\bm\psi}^{\mfk e}_v}(f_v^{\mbf G^{\mfk e}}).
\end{equation}
So the assertion follows.
\end{proof}

\section{Discrete case}\label{lsisieiutythrenfmess}

In this section, we prove the main theorem when $\psi^{\mfk e}$ is discrete and good. In particular, we can write $\psi^{\mfk e}=\psi_1+\psi_2$, and write
\begin{equation*}
\psi^\GL_1=\phi_1\boxtimes\sp_{a_1}+\cdots+\phi_r\boxtimes\sp_{a_r}, \quad \psi^\GL_2=\phi_{r+1}\boxtimes\sp_{a_{r+1}}+\cdots+\phi_k\boxtimes\sp_{a_k},
\end{equation*}
where $a_i$ is odd for each $1\le i\le k$, and
\begin{equation*}
a_1<a_2\cdots<a_r, \quad a_{r+1}<a_{r+2}<\cdots<a_k. 
\end{equation*}
Here we have grouped together the summands with the same $\SL_2$-factor, so that the integers $a_i$ are pairwise distinct. We assume that $G$ is not quasisplit. In particular, $n\ge 2$ and $\disc(V)=1$. To prove the main theorem, firstly we globalize the local field $K$ to a totally real number field:

\begin{lm}\label{sisisenfiesm}
For any non-Archimedean local field $K$, there exists a totally real number field $F$ with finite places $v_0, v_1\in \fPla_F$ such that  $F_{v_0}\cong F_{v_1}\cong K$.
\end{lm}
\begin{proof}
This is an easy application of Krasner's lemma, see for example the proof of \cite{Art13}*{Lemma 6.2.1} and \cite{Ish24}*{Lemma 6.2}.
\end{proof}

Fix such a totally real number field $F$ as in Lemma~\ref{sisisenfiesm}, together with a fixed embedding $\tau_0: F\to \bb R$. We now globalize the endoscopic datum $\mfk e$. Suppose $G^{\mfk e}=\SO(2n')^{\mfk d}\times\SO(2n'')^{\mfk d}$ with $n'\ge n''\ge 0$. Note that each $\phi_i$ is a discrete tempered $A$-parameter, i.e. a discrete $L$-parameter, for some even orthogonal group $\SO(2n_i)^{\mfk d_i}$, such that 
\begin{equation*}
\prod_{i=1}^r\mfk d_i=\prod_{i=r+1}^{k}\mfk d_i=\mfk d.
\end{equation*}

We choose elements $\mfk d_1^0, \ldots, \mfk d_k^0\in K^\times/(K^\times)^2$ and pairwise nonisomorphic local simple supercuspidal $L$-parameters $\phi^0_i$ for $\SO(2n_i)^{\mfk d_i^0}$ for $1\le i\le k$. This is possible by \cite[Theorem~C.3]{C-Z24}. Then we choose elements $\mfk D_1, \ldots, \mfk D_k\in F^\times$ 
such that  
\begin{enumerate}
\item
$\mfk D_i=\mfk d_i\in F_{v_1}^\times/(F_{v_1}^\times)^2$ for each $1\le i\le k$;
\item
$\mfk D_i=\mfk d_i^0\in F_{v_0}^\times/(F_{v_0}^\times)^2$ for each $1\le i\le k$;
\item
$(-1)^{n_i}\mfk D_i$ is totally positive for each $1\le i\le k$;
\item
For any two finite subsets $S, T$ of $\{1, 2,\ldots, k\}$, if $\prod_{i\in S}\mfk D_i=\prod_{i\in T}\mfk D_i\in F^\times/(F^\times)^2$, then $S=T$.
\end{enumerate}
This is possible by the weak approximation theorem.

We define
\begin{equation*}
\mfk D'\defining \prod_{i=1}^r\mfk D_i, \quad \mfk D''\defining \prod_{i=r+1}^k\mfk D_i,
\end{equation*}
and choose a quadratic space $\mbf V$ over $F$ of rank $2n$ with discriminant $\mfk D=\mfk D'\mfk D''$, such that 
\begin{enumerate}
\item
$\mbf V$ is positive definite at each real place of $F$ except for $\tau_0$; 
\item
For each finite place $v$ of $F$ except for $v=v_1$, the Hasse--Witt invariant $\eps(\mbf V\otimes_FF_v)$ is $1$; and
\item
$\eps(\mbf V\otimes_FF_{v_1})=-1$.
\end{enumerate}
Note that such a quadratic space exists by the Hasse--Minkowski theorem \cite{Gro21}*{Theorem 2.1}. We write $\mbf G=\SO(\mbf V)$. Then there exists an elliptic endoscopic datum $\mfk E$ for $\mbf G$ such that  $\mbf G^{\mfk E}\cong \SO(2n')^{\mfk D'}\times \SO(2n'')^{\mfk D''}$.

By the choices of $\mfk D_1, \ldots, \mfk D_k$, we have achieved the following:
\begin{enumerate}
\item
$\mbf G$ is anisotropic;
\item
Both $\mbf G^{\mfk E}(F\otimes\bb R)$ and $\mbf G(F\otimes\bb R)$ admit representations in the discrete series; and 
\item
$\mbf G\otimes_FF_v$ is quasisplit for each finite place $v$ of $F$ except for $v_1$, and $\mbf G\otimes_FF_{v_1}\cong G$.
\end{enumerate}
Let $\mbf G^*=\SO(2n)^{\mfk D}$ be the unique quasisplit inner form of $\mbf G$. We may fix a pure inner twisting $(\varrho, z)$ between $\mbf G^*$ and $\mbf G$, which localizes to a pure inner twist $(\varrho_v, z_v)$ between $\mbf G^*_v$ and $\mbf G_v$ for each place $v$ of $F$. We further assume that $(\varrho_{v_0}, z_{v_0})$ is a trivial pure inner twist.

We now globalize the local $A$-parameter $\psi^{\mfk e}$. This is a standard step; see, for example, \cite[Corollary~6.2.3]{Art13}, \cite[Remark~3 after Corollary~6.2.4]{Art13}, and \cite[Lemma~4.3.1]{KMSW}.

\begin{lm}\label{glaobidninssAprieries}
There exist elliptic global $A$-parameters $\bm\psi', \bm\psi''$ for $\SO(2n')^{\mfk D'}$ and $\SO(2n'')^{\mfk D''}$, respectively, such that  
\begin{enumerate}
\item
\begin{equation*}
\bm\psi'=\bm\psi_1\boxtimes\sp_{a_1}+\cdots+\bm\psi_r\boxtimes\sp_{a_r}, \quad \bm\psi''=\bm\psi_{r+1}\boxtimes\sp_{a_{r+1}}+\cdots+\bm\psi_k\boxtimes\sp_{a_k},
\end{equation*}
where $\bm\psi_i$ is an irreducible self-dual cuspidal automorphic representation of $\GL(2n_i, \Ade_F)$ with central character equal to the quadratic character of $\mbf C_F$ associated with $\mfk D_i$, for each $1\le i\le k$;
\item
$\tilde{\bm\psi}_{i, v_0}=\phi^0_i$ for each $1\le i\le k$;
\item
$\tilde{\bm\psi}_{i, v_1}=\phi_i$ for each $1\le i\le k$; and
\item
$\bm\psi=\bm\psi'+\bm\psi''$ is an elliptic global $A$-parameter for $\mbf G$. Moreover, for any real place $\tau$ of $F$, the localization $\tilde{\bm\psi}_\tau$ is a $\varsigma$-equivalence class of Adams--Johnson parameters for $\mbf G\otimes_{F, \tau}\bb R$.
\end{enumerate}
\end{lm}
\begin{proof}
For each $1\le i\le k$, let $\mbf G_i$ denote the quasisplit global even special orthogonal group $\SO(2n_i)^{\mfk D_i}$ over $F$. Fix, for each $1\le i\le k$, an irreducible representation $\pi_i$ of $\mbf G_i(F_{v_1})$ with $L$-parameter $\phi_i$, and an irreducible representation $\pi_i^0$ of $\mbf G_i(F_{v_0})$ with $L$-parameter $\phi_i^0$. Since $\mbf G_i$ is anisotropic and $\mbf G(F\otimes\bb R)$ admits discrete series, the standard Plancherel density theorem, for example \cite{Shi12}*{Theorem~1.1.(\rmnum1)}, applies. In particular, if we choose an algebraic irreducible representation $\xi_i$ of $\paren{\Res_{F/\bb Q}\mbf G_i}\otimes\bb C$, then it follows from \cite{Shi12}*{Theorem~1.1.(\rmnum1)} that there exists a $\xi_i$-cohomological cuspidal automorphic representations $\Pi_i$ of $\mbf G_i(\Ade_F)$ such that $\Pi_{i, v_0}\cong \pi_i^0$ and $\Pi_{i, v_1}\cong\pi_i$.

For each $1\le i\le k$, Arthur's multiplicity formula Theorem~\ref{mianieijidinfies} implies that there exists an elliptic global $A$-parameter $\bm\psi_i$ for $\mbf G_i$ such that $\Pi_i$ belongs to the corresponding packet, i.e. $\Pi_i\in \mcl A_{2, \bm\psi_i}(\mbf G_i)$. Since the localization $\bm\psi_{i, v_0}$ is isomorphic to $\phi_i^0$, and $\phi_i^0$ is simple supercuspidal, it follows that $\bm\psi_{i, v_0}$ must be in fact an irreducible self-dual cuspidal automorphic representation of $\GL(2n_i, \Ade_F)$ with central character equal to the quadratic character of $\mbf C_F$ associated with $\mfk D_i$. We then define elliptic global $A$-parameters
\begin{equation*}
\bm\psi'=\bm\psi_1\boxtimes\sp_{a_1}+\cdots+\bm\psi_r\boxtimes\sp_{a_r}, \quad \bm\psi''=\bm\psi_{r+1}\boxtimes\sp_{a_{r+1}}+\cdots+\bm\psi_k\boxtimes\sp_{a_k}.
\end{equation*}
for $\SO(2n')^{\mfk D'}$ and $\SO(2n'')^{\mfk D''}$, respectively. These parameters satisfy Conditions~(1)--(3).

Finally, by choosing algebraic irreducible representations $\xi_i$ of $\paren{\Res_{F/\bb Q}\mbf G_i}\otimes\bb C$ with highest weights sufficiently regular and sufficiently far apart from one another, we may ensure that, for each real place $\tau$ of $F$, the localizations $\tilde{\bm\psi}_\tau$ has $C$-algebraic and regular infinitesimal characters. Hence $\bm\psi=\bm\psi'+\bm\psi''$ also satisfies Condition~(4).
\end{proof}

We choose a global elliptic parameter $\bm\psi$ for $\mbf G$ satisfying the requirement of Lemma~\ref{glaobidninssAprieries}. Then the natural map
\begin{equation*}
\mfk S_{\bm\psi}\to \mfk S_{\tilde{\bm\psi}_{v_0}}
\end{equation*}
is an isomorphism. Consider $f=\otimes_vf_v\in\tilde{\mcl H}(\mbf G(\Ade_F))$ such that, for each $\tilde\pi_{v_0}\in\tilde\Pi_{\tilde{\bm\psi}_{v_0}}(\mbf G_{v_0})$,
\begin{equation}\label{lsieinfeiuhfniemms}
\tilde\Theta_{\tilde\pi_{v_0}}(f_{v_0})\cdot\iota_{\mfk w_{v_0}, z_{v_0}}(\tilde\pi_{v_0})(s^{\mfk E}s_{\tilde{\bm\psi}_{v_0}})=1,
\end{equation}
and the remaining components $f_v$ for $v\in\Pla_F\setm\{v_0\}$ are arbitrary. This is possible because $\tld{\bm\psi}_{v_0}$ is elementary, so by Lemma~\ref{isisieiemfiems} the map
\begin{equation*}
\iota_{\mfk w_{v_0}, z_{v_0}}: \tilde\Pi_{\tilde\psi}(G)\cong \Irr(\ovl{\mfk S}_{\tilde\psi}).
\end{equation*}
is a bijection. In particular, the characters $\tld\Theta_{\tilde\pi_{v_0}}$ are linearly independent as functions on $\tilde{\mcl H}(\mbf G(F_{v_0}))$. Hence such an $f_{v_0}$ exists.

We apply Equation~\eqref{ismsiniefmeis} to $f$. For each $\mfk E'\in \mcl E_\ellip(\mbf G)$, let $f^{\mbf G^{\mfk E'}}\in\tilde{\mcl H}(\mbf G^{\mfk E'}(\Ade_F))$ be an adelic transfer of $f$ to $\mbf G^{\mfk E'}$. It follows from the definition of $f_{v_0}$ (Equation~\eqref{lsieinfeiuhfniemms}) and the main theorem in the quasisplit case (which is known, see \S\ref{lslsieieurheifmeis}) that the term on the right-hand side of Equation~\eqref{ismsiniefmeis} corresponding to $\mfk E'$ vanishes unless $\mbf G^{\mfk E'}_{v_0}\cong\mbf G^{\mfk E}_{v_0}$. Indeed,
\begin{align*}
\tilde\Theta_{\tilde{\bm\psi}^{\mfk E'}_{v_0}}(f_{v_0}^{\mbf G^{\mfk E'}})
&=\sum_{\tilde\pi_{v_0}\in\tld\Pi_{\tld{\bm\psi}_{v_0}}(\mbf G_{v_0})}\iota_{\mfk w_{v_0}, z_{v_0}}(\tilde\pi_{v_0})(s^{\mfk E'}s_{\tld{\bm\psi}})\cdot\tilde\Theta_{\tilde\pi_{v_0}}(f_{v_0})\\
&=\sum_{\tilde\pi_{v_0}\in\tld\Pi_{\bm\psi_{v_0}}(\mbf G_{v_0})}\iota_{\mfk w_{v_0}, z_{v_0}}(\tilde\pi_{v_0})(s^{\mfk E'}\cdot s^{\mfk E}).
\end{align*}
By Lemma~\ref{isisieiemfiems} and orthogonality relation, this expression vanishes unless $s^{\mfk E'}=s^{\mfk E}$ or $s^{\mfk E'}=s^{\mfk E}\cdot z_{\tilde\psi_{v_0}}$, or equivalently $\mbf G^{\mfk E'}_{v_0}\cong \mfk G^{\mfk E}_{v_0}$.

Moreover, in view of the shape of $\bm\psi$ and our choices of $\mfk D_i$, the terms on the right-hand side of Equation~\eqref{ismsiniefmeis} vanishes unless $\mbf G^{\mfk E'}\cong \mbf G^{\mfk E}$.\footnote{This is the point at which we must restrict to good parameters.} Then it follows from the simple stable trace formula Equation~\eqref{simplestaleitneirnies} and linear independence of characters at unramified places of $F$ that
\begin{equation}\label{isnsieneifes}
\sum_{\tilde\Pi}\prod_{v\in\Pla}\tilde\Theta_{\tilde\Pi_v}(f_v)=\iota(\mfk E)\frac{m_{\bm\psi^{\mfk E}}\ve_{\bm\psi^{\mfk E}}(s_{\bm\psi^{\mfk E}})}{m_{\bm\psi}\#\ovl{\mfk S}_{\bm\psi^{\mfk E}}}\prod_{v\in\Pla}\tilde\Theta_{\tilde{\bm\psi}^{\mfk E}_v}(f^{\mbf G^{\mfk E}}_v)=\frac{\ve_{\bm\psi^{\mfk E}}(s_{\bm\psi^{\mfk E}})}{\#\ovl{\mfk S}_{\bm\psi}}\prod_{v\in\Pla}\tilde\Theta_{\tilde{\bm\psi}^{\mfk E}_v}(f^{\mbf G^{\mfk E}}_v).
\end{equation}
Here $\Pla$ is a finite set of places of $F$ containing $\infPla_F\cup\{v_0, v_1\}$ such that, for each finite place $v$ of $F$ outside $\Pla$,
\begin{itemize}
\item
$\tilde\phi_{\tilde{\bm\psi}_v}$ is unramified;
\item
the inner twist $(\varrho_v, z_v)$ is trivial; and
\item
$f_v=\uno_{\mdc K_v}$ is the characteristic function of the fixed $\varsigma$-invariant hyperspecial subgroup $\mdc K_v$ of $\mbf G(F_v)$.
\end{itemize}
Such a choice is possible by the definition of $\tilde{\mcl H}(\mbf G(\Ade_F))$. With this choice, the sum on the left-hand side runs through all discrete automorphic representations $\tilde\Pi\in \tld{\Irr}_\disc(\mbf G(\Ade_F))$ such that  
\begin{enumerate}
\item
$\tilde\Pi_v$ is unramified with $L$-parameter $\tilde\phi_{\tilde{\bm\psi}_v}$ for any $v\in\fPla_F\setm\Pla$; and
\item
$\tilde\Pi_v\in \tilde\Pi_{\tilde{\bm\psi}_v}(\mbf G\otimes_FF_v)$ for each $v\in\Pla$.
\end{enumerate}
(see~\cite{C-G15}*{\S11.5} for a similar argument).

We now analyze the contribution to both sides of Equation~\eqref{isnsieneifes} for each $v\in\Pla\setm\{v_1\}$.
\begin{itemize}
\item
Suppose $v=v_0$. Then
\begin{equation}\label{ismsineimfeis}
\tilde\Theta_{\tilde\Pi_{v_0}}(f_{v_0})\cdot\iota_{\mfk w_{v_0}, z_{v_0}}(\tilde\Pi_{v_0})(s^{\mfk E}s_{\tilde{\bm\psi}_{v_0}})=1
\end{equation}
by the choice of $f_{v_0}$. On the other hand, by the main theorem in the quasisplit case,
\begin{align}\label{ismsiemfiems}
\tilde\Theta_{\tilde{\bm\psi}^{\mfk E}_{v_0}}(f^{\mbf G^{\mfk E}}_{v_0})&=\sum_{\tilde\pi_{v_0}\in\tilde\Pi_{\tilde{\bm\psi}_{v_0}}(\mbf G\otimes_FF_{v_0})}\iota_{\mfk w_{v_0}, z_{v_0}}(\tilde\pi_{v_0})(s^{\mfk E}s_{\tilde{\bm\psi}_{v_0}})\tilde\Theta_{\tilde\pi_{v_0}}(f_{v_0})\\
&=\#\tilde\Pi_{\tilde{\bm\psi}_{v_0}}(\mbf G\otimes_FF_{v_0})\\
&=\#\ovl{\mfk S}_{\bm\psi}.
\end{align}
Here for the last equality we used Equation~\eqref{isisieiemfiems}, noting that $\tilde{\bm\psi}_{v_0}$ is elementary.
\item
Suppose $v\in\Pla\setm\paren{\infPla_F\cup\{v_0, v_1\}}$. As $\mbf G_v$ is quasisplit, it follows from the main theorem in the quasisplit case that
\begin{equation}\label{ismsieifeniefnsiws}
\sum_{\tilde\pi_v\in\tilde\Pi_{\tilde{\bm\psi}_v}(\mbf G\otimes_FF_v)}\iota_{\mfk w_v, z_v}(\tilde\pi_v)(s^{\mfk E}s_{\tilde{\bm\psi}_v})\tilde\Theta_{\tilde\pi_v}(f_v)=\tilde\Theta_{\tilde{\bm\psi}^{\mfk E}_{v_1}}(f^{\mbf G^{\mfk E}}_v).
\end{equation}
\item
Suppose $\tau\in\infPla_F$. As $\tilde{\bm\psi}_\tau$ is an Adams--Johnson parameter, by Theorem~\ref{ismsienfieimfws}, we also have
\begin{equation}\label{ismiiwiwiwoqq}
\sum_{\tilde\pi_\tau\in\tilde\Pi_{\tilde{\bm\psi}_\tau}(G\otimes_{F, \tau}\bb R)}\iota_{\mfk w_\tau, z_\tau}(\tilde\pi_\tau)(s^{\mfk E}s_{\tilde{\bm\psi}_\tau})\tilde\Theta_{\tilde\pi_\tau}(f_\tau)=\tilde\Theta_{\tilde{\bm\psi}^{\mfk E}_\tau}(f_\tau^{\mbf G^{\mfk E}}).
\end{equation}
\end{itemize}

For each $\tilde\Pi$ appearing in the left-hand side of Equation~\eqref{isnsieneifes}, it follows from Arthur's multiplicity formula Equation~\eqref{mianieijidinfies} that
\begin{equation}\label{ismsieifniesws}
\prod_{v\in\Pla}\iota_{\mfk w_v, z_v}(\tilde\Pi_v)(s^{\mfk E}s_{\bm\psi})=\ve_{\bm\psi}(s^{\mfk E}s_{\bm\psi}).
\end{equation}
Conversely, it follows from Lemma~\ref{isisieiemfiems} that for any
\begin{equation*}
\btimes_{v\in\Pla\setm\{v_0\}}\tilde\Pi^\dagger_v\in\prod_{v\in\Pla\setm\{v_0\}}\tilde\Pi_{\tilde{\bm\psi}_v}(\mbf G\otimes_FF_v),
\end{equation*}
there exists a unique $\tilde\Pi\in\tld\Irr_\disc(\mbf G(\Ade_F))$ appearing in the left-hand side of Equation~\eqref{isnsieneifes} with $\prod_{v\in\Pla\setm\{v_0\}}\tilde\Pi_v\cong \prod_{v\in\Pla\setm\{v_0\}}\tilde\Pi^\dagger_v$.

Finally, it follows from \cite{Art13}*{Lemma 4.4.1} that
\begin{equation}\label{eiiwiwimwmwww}
\ve_{\bm\psi^{\mfk E}}(s_{\bm\psi^{\mfk E}})=\ve_{\bm\psi}(s^{\mfk E}s_{\bm\psi}).
\end{equation}
So we conclude from Equations~\eqref{isnsieneifes}, \eqref{ismsineimfeis}, \eqref{ismsiemfiems}, \eqref{ismsieifeniefnsiws}, \eqref{ismiiwiwiwoqq}, \eqref{ismsieifniesws}, and \eqref{eiiwiwimwmwww} that
\begin{equation*}
\sum_{\tilde\pi_{v_1}\in\tilde\Pi_{\tilde{\bm\psi}_{v_1}}(\mbf G\otimes_FF_{v_1})}\iota_{\mfk w_{v_1}, z_{v_1}}(\tilde\pi_{v_1})(s^{\mfk e}s_{\tilde{\bm\psi}_{v_1}})\tilde\Theta_{\tilde\pi_{v_1}}(f_{v_1})=\tilde\Theta_{\tilde{\bm\psi}_{v_1}^{\mfk e}}(f^{\mbf G^{\mfk e}}_{v_1}),
\end{equation*}
which proves the main theorem for $\tilde\psi=\tilde{\bm\psi}_{v_1}$.

\section{Non-discrete case}\label{non-disniniieheins}

In this section, suppose $K$ is non-Archimedean. We prove the main theorem for local $A$-parameters $\psi\in\Psi(G)$ that factors through an good local $A$-parameter $\psi^{\mfk e}$ of $G^{\mfk e}$ that is not discrete. 
We write $\mfk d=\disc(V)$ for the discriminant of $V$.


In this case there is a properly contained Levi subgroup $M_{\mfk e}$ of $G^{\mfk e}$ such that $\psi^{\mfk e}$ factors through an good discrete $A$-parameter for $M_{\mfk e}$. Note that $M_{\mfk e}$ is of the form
\begin{equation*}
M_{\mfk e}=\paren{\SO(2m_1)^{\mfk d_1}\times \prod_{i\in I_1}\GL(n_i)}\times \paren{\SO(2m_2)^{\mfk d_2}\times \prod_{i\in I_2}\GL(n_i)},
\end{equation*}
with $\mfk d_1\mfk d_2=\mfk d\in K^\times/(K^\times)^2$. So the $A$-parameter $\psi$ factors through an $A$-parameter $\psi_M$ of a Levi subgroup $\hat{M^*}$ of $\hat G$, where $M^*$ is a standard Levi subgroup of $G^*$ such that 
\begin{equation*}
M^*\cong\SO(2n_0)^{\mfk d}\times\prod_{i\in I}\GL(n_i)
\end{equation*}
with $I=I_1\coprod I_2$ and $n_0=m_1+m_2$. We may write
\begin{equation*}
\psi^\GL=\psi_\tau+\psi_0^\GL+\psi_\tau^\vee,
\end{equation*}
where $\psi_\tau$ is a representation of $\WD_K\times\SL(2, \bb C)$ corresponding to an admissible irreducible unitarizable representation $\tau=\prod_{i\in I}\tau_i$ of $\prod_{i\in I}\GL(n_i)$ of Arthur type, and $\psi_0$ is a discrete parameter for $\SO(2n_0)^{\mfk d}$. We write $M=\varrho(M^*)$, which is a Levi subgroup of $G$. Then $M_{\mfk e}$ is an endoscopic group for $M$, i.e. there exists an endoscopic triple $\mfk e_0$ for $M$ such that  $M^{\mfk e_0}=M_{\mfk e}$. We may choose $\mfk e_0$ such that, under the commutative diagram
\begin{equation*}
\begin{tikzcd}[sep=large]
S_{\psi_M}\ar[r, hookrightarrow, "\iota"]\ar[d] & S_\psi\ar[d]\\
\mfk S_{\psi_M}\ar[r, hookrightarrow, "\ovl\iota"] & \mfk S_\psi
\end{tikzcd},
\end{equation*}
the image of $s^{\mfk e_0}\in Z_{\hat M}(\psi_M)$ in $\mfk S_{\psi_M}= \pi_0(Z_{\hat M}(\psi_M))$ is mapped by $\ovl\iota$ to the image of $s^{\mfk e}\in S_\psi$ in $\mfk S_\psi$.

We recall the following lemma on descent property of transfer pairings.

\begin{lm}\label{ismsieiws}\enskip
\begin{enumerate}
\item
For each $\gamma\in M^{\mfk e_0}(K)$ and $\delta\in M(K)$, the following identity holds:
\begin{equation*}
\Delta[\mfk w_0, \mfk e_0, z](\gamma, \delta)=\Delta[\mfk w, \mfk e, z](\gamma, \delta)\cdot\largel{\frac{D_{M^{\mfk e_0}}^{G^{\mfk e}}(\gamma)}{D_M^G(\delta)}}^{\frac{1}{2}}.
\end{equation*}
Here $\mfk w_0$ is the induced Whittaker datum on $M^*$, and $D_{M^{\mfk e_0}}^{G^{\mfk e}}(\gamma)$ and $D_M^G(\delta)$ are the relative Weyl discriminants.
\item
If $f\in \mcl H(G)$ and $f^{\mfk e}\in \mcl H(G^{\mfk e})$ are $\Delta[\mfk w, \mfk e, z]$-matching, then their constant terms $f_M\in \mcl H(M)$ and  $f^{\mfk e}_{M^{\mfk e_0}}\in \mcl H(M^{\mfk e_0})$ is $\Delta[\mfk w_0, \mfk e_0, z]$-matching.
\end{enumerate}
\end{lm}
\begin{proof}
The proof is similar to that of \cite{KMSW}*{Lemma 1.1.4}.
\end{proof}

We now consider the $A$-packets corresponding to $\psi^{\mfk e}$ and $\psi$. It follows from Theorem~\ref{psmisemifniefms} that
\begin{equation*}
\tilde\Pi_{\psi^{\mfk e}}(G^{\mfk e})=\coprod_{\tilde\sigma\in\tilde\Pi_\psi(M^{\mfk e_0})}\{\text{irreducible components of }\bx I_{M^{\mfk e_0}}^{G^{\mfk e}}(\tilde\sigma)\}
\end{equation*}
and
\begin{equation*}
\tilde\Pi_{\psi}(G)=\coprod_{\tilde\pi\in\tilde\Pi_{\psi_M}(M)}\{\text{irreducible components of }\bx I_M^G(\tilde\pi)\}.\footnote{Note that $M$ is a product of a special orthogonal group and general linear groups, so the $A$-packets and also the pairings $\iota_{\mfk w_0, z}$ are obviously defined.}
\end{equation*}
Here the induced representations are defined by chosen representatives of $\tilde\sigma$ or $\tilde\pi$, and they are multiplicity-free sums of irreducible admissible representations whose $\varsigma$-equivalence classes are independent of the representatives chosen. Hence
\begin{equation}\label{ismsieiowpw}
\sum_{\tilde\pi\in\tilde\Pi_\psi(G)}\iota_{\mfk w, z}(\tilde\pi)(s^{\mfk e}s_\psi)\tilde\Theta_{\tilde\pi}=\sum_{\tilde\pi\in\tilde\Pi_\psi(G)}\iota_{\mfk w, z}(\tilde\pi)(\iota(s^{\mfk e_0}s_{\psi_M}))\tilde\Theta_{\tilde\pi}=\sum_{\tilde\pi\in\tilde\Pi_{\psi_M}(M)}\iota_{\mfk w_0, z}(\tilde\pi)(s^{\mfk e_0}s_{\psi_M})\tilde\Theta_{\bx I_M^G(\tilde\pi)}.
\end{equation}

It follows from Equation~\eqref{ismsieiowpw} and the parabolic descent formula that
\begin{equation*}
\sum_{\tilde\pi\in\tilde\Pi_\psi(G)}\iota_{\mfk w, z}(\tilde\pi)(s^{\mfk e}s_\psi)\tilde\Theta_{\tilde\pi}(f)=\sum_{\tilde\pi\in\tilde\Pi_{\psi_M}(M)}\iota_{\mfk w_0, z}(\tilde\pi)(s^{\mfk e_0}s_{\psi_M})\tilde\Theta_{\tilde\pi}(f_M).
\end{equation*}
Similarly, it follows from the parabolic descent formula that
\begin{equation*}
\tilde\Theta_{\psi^{\mfk e}}(f^{\mfk e})=\tilde\Theta_{\psi^{\mfk e}}(f^{\mfk e}_{M^{\mfk e_0}}).
\end{equation*}
Here we note that $f^{\mfk e}_{M^{\mfk e_0}}$ is $\varsigma$-invariant. 


We now prove the main theorem Theorem~\ref{thm main}. By Lemma~\ref{ismsieiws}, $f^{\mfk e}_{M^{\mfk e_0}}$ and $f_M$ are $\Delta[\mfk w_0, \mfk e_0, z]$-matching test functions. Thus, it follows from the main theorem in the discrete case (which is known, see~\S\ref{lsisieiutythrenfmess})  that
\begin{equation*}
\tilde\Theta_{\psi^{\mfk e}}(f^{\mfk e})=\tilde\Theta_{\psi^{\mfk e}}(f^{\mfk e}_{M^{\mfk e_0}})=\sum_{\tilde\pi\in\tilde\Pi_{\psi_M}(M)}\iota_{\mfk w_0, z}(\tilde\pi)(s^{\mfk e_0}s_{\psi_M})\tilde\Theta_{\tilde\pi}(f_M)=\sum_{\tilde\pi\in\tilde\Pi_\psi(G)}\iota_{\mfk w, z}(\tilde\pi)(s^{\mfk e}s_\psi)\tilde\Theta_{\tilde\pi}(f).
\end{equation*}
This concludes the proof when the good $A$-parameter $\psi^{\mfk e}$ is not discrete.

\bibliographystyle{plain}
\bibliography{bibliography}

\vspace{2em}
\noindent\textsc{Hao Peng}\\
Department of Mathematics, Massachusetts Institute of Technology, Cambridge, MA 02139, USA\\
\textit{Email}: \texttt{hao\_peng@mit.edu}
    
\end{document}